\documentclass[reqno]{article}

\usepackage[T1]{fontenc}
\usepackage[utf8]{inputenc}
\usepackage[english]{babel}

\usepackage{amsmath}
\usepackage{geometry}
\usepackage{color}
\usepackage{xcolor}
\usepackage{enumitem}
\usepackage[%
	linktoc=all,
]{hyperref}

\usepackage{graphicx}
\usepackage{wrapfig}
\usepackage{caption}
\usepackage{float}

\usepackage{mathtools}
\usepackage{amssymb}
\usepackage{wasysym}
\usepackage{textcomp}
\usepackage{stmaryrd}
\usepackage{esint}

\usepackage{amsthm}
\usepackage{thmtools}

\usepackage{xspace}

\colorlet{myred}{red!50!black}
\colorlet{mylightblue}{blue!50!black}
\colorlet{mydarkblue}{blue!80!black}
\colorlet{mygreen}{green!50!black}

\hypersetup{%
	colorlinks = true,
	linkcolor = myred,
	citecolor = mylightblue,
	urlcolor = mydarkblue
}

\setlist[itemize,1]{label=\ensuremath{\diamond}} 
\setlist[enumerate,1]{label=(\textit{\roman*})}

\newcommand{\guillemet}[1]{``#1''}

\newcommand*{\N}{\mathbb{N}}

\newcommand*{\R}{\mathbb{R}}
\newcommand*{\C}{\mathbb{C}}
\renewcommand*{\d}{\mathrm{d}}
\newcommand*{\id}[1][{}]{1_{#1}}
\newcommand*{\zero}[1][{}]{0_{#1}}
\newcommand*{\deq}{\overset{\scriptscriptstyle{\mathrm{def}}}{=}}
\newcommand*{\lra}{\longrightarrow}

\newcommand*{\ri}{\ensuremath\mathrm{i}}
\newcommand*{\re}{\ensuremath\mathrm{e}}

\newcommand*{\Norm}[1]{\left\lVert#1\right\rVert}
\newcommand*{\norm}[1]{\lVert#1\rVert}
\newcommand*{\Absolute}[1]{\left\lvert#1\right\rvert}
\newcommand*{\absolute}[1]{\lvert#1\rvert}

\newcommand*{\scalp}[1]{\langle #1\rangle}

\newcommand*{\jump}[1]{[#1]}

\newcommand*{\transpose}[1]{#1^\mathrm{T}}
\newcommand*{\Cal}[1]{\mathcal{#1}}

\DeclareMathOperator*{\real}{\mathrm{Re}}
\DeclareMathOperator*{\imag}{\mathrm{Im}}
\DeclareMathOperator*{\range}{range}
\DeclareMathOperator*{\Span}{Span}

\DeclareMathOperator*{\trace}{Tr}

\DeclareMathOperator*{\Bigo}{\mathcal{O}}
\newcommand*{\bigo}{\mathcal{O}}

\newcommand*{\Set}[1]{\left\lbrace #1 \right\rbrace}
\newcommand*{\buc}[1]{\ensuremath{{BUC}^{#1}}}

\newcommand*{\spectrum}[1][]{\varSigma_\text{#1}}

\newcommand*{\uu}{{\ensuremath{\underline{u}}}}
\newcommand*{\dual}[1]{\tilde{#1}}

\newcounter{decomposition}
\newcommand*{\nombre}[1]{%
	\setcounter{decomposition}{#1}%
	\textit{\Roman{decomposition}}%
}
\newcommand{\ki}{k^{\nombre{1}}}
\newcommand{\Ki}{K^{\nombre{1}}}
\newcommand{\Kii}{K^{\nombre{2}}}
\newcommand{\Kiii}{K^{\nombre{3}}}
\newcommand{\Kiv}{K^{\nombre{4}}}
\newcommand{\gi}{g^{\nombre{1}}}
\newcommand{\Gi}{G^{\nombre{1}}}
\newcommand{\Gii}{G^{\nombre{2}}}
\newcommand{\Giii}{G^{\nombre{3}}}
\newcommand{\Giv}{G^{\nombre{4}}}
\newcommand{\si}{s^{\nombre{1}}}
\newcommand{\Si}{S^{\nombre{1}}}
\newcommand{\Sii}{S^{\nombre{2}}}
\newcommand{\Siii}{S^{\nombre{3}}}
\newcommand{\Siv}{S^{\nombre{4}}}
\newcommand{\vp}{v^+_{1}}
\newcommand{\dvp}{\tilde{v}^+_{1}}
\newcommand{\pip}{\pi^+_{1}}

\newcommand{\cA}{\mathcal{A}}
\newcommand{\cC}{\mathcal{C}}
\newcommand{\cL}{\mathcal{L}}
\newcommand{\cN}{\mathcal{N}}
\newcommand{\cO}{\mathcal{O}}
\newcommand{\cT}{\mathcal{T}}

\newcommand{\nup}[1]{\nu_{#1}^+}
\newcommand{\num}[1]{\nu_{#1}^-}
\newcommand{\nupm}[1]{\nu_{#1}^\pm}
\newcommand{\phip}[1]{\phi_{#1}^+}
\newcommand{\phim}[1]{\phi_{#1}^-}
\newcommand{\phipm}[1]{\phi_{#1}^\pm}
\newcommand{\psip}[1]{\psi_{#1}^+}

\newcommand{\psipm}[1]{\psi_{#1}^\pm}
\newcommand{\Phip}[1]{\varPhi_{#1}^+}
\newcommand{\Phim}[1]{\varPhi_{#1}^-}
\newcommand{\Phipm}[1]{\varPhi_{#1}^\pm}
\newcommand{\Psip}[1]{\varPsi_{#1}^+}
\newcommand{\Psim}[1]{\varPsi_{#1}^-}
\newcommand{\Psipm}[1]{\varPsi_{#1}^\pm}
\newcommand{\dphip}[1]{\dual{\phi}_{#1}^+}
\newcommand{\dphim}[1]{\dual{\phi}_{#1}^-}
\newcommand{\dphipm}[1]{\dual{\phi}_{#1}^\pm}

\newcommand{\dpsipm}[1]{\dual{\psi}_{#1}^\pm}
\newcommand{\dPhip}[1]{\dual{\varPhi}_{#1}^+}
\newcommand{\dPhim}[1]{\dual{\varPhi}_{#1}^-}
\newcommand{\dPhipm}[1]{\dual{\varPhi}_{#1}^\pm}
\newcommand{\dPsip}[1]{\dual{\varPsi}_{#1}^+}
\newcommand{\dPsim}[1]{\dual{\varPsi}_{#1}^-}
\newcommand{\dPsipm}[1]{\dual{\varPsi}_{#1}^\pm}

\newcommand{\Fp}[1]{F_{#1}^+}
\newcommand{\Fm}[1]{F_{#1}^-}

\newcommand*{\ind}[1]{\mathbf{1}_{#1}}

\newcommand{\K}{\mathcal{K}}
\newcommand{\Dirac}{\boldsymbol{\delta}}
\newcommand{\ac}{\mathrm{ac}}
\newcommand{\at}{\mathrm{at}}

\newtheorem{theorem}{Theorem}[section]

\newtheorem{proposition}[theorem]{Proposition}
\newtheorem{lemma}[theorem]{Lemma}

\theoremstyle{definition}
\newtheorem{definition}[theorem]{Definition}
\newtheorem{assumption}{Assumption}

\theoremstyle{remark}
\newtheorem{remark}[theorem]{Remark}

\title{Orbital stability of monostable waves \\for reaction-diffusion systems}
\author{L. Garénaux}
\date{\today}

\begin{document}

\maketitle

\begin{abstract}
We study stability of monostable waves for reaction-diffusion systems. When the solution is initially close to a fast wave profile in optimal topology, we prove convergence to a shifted profile. The proof relies on explicit resolvent kernels estimates, allowing to handle weakly localized perturbations. It allows phase shift construction even when the translational eigenvalue is not associated to a zero of the Evans function. 

We further discuss distinction between Evans and Fourier eigenmodes when the marginal group velocity are directed towards the wave interface.\\[0.5em]
\end{abstract}

\section{Introduction}

\subsection{Litterature}
In the present work, we are interested in one dimensional parabolic systems
\begin{equation}
\label{e:main}
u_t = d u_{xx} + \left(f(u)\right)_x + g(u),
\hspace{4em}
x \in \R,
\end{equation}
with unknown $u(t,x) \in \C^n$, a diagonal matrix $d\in \C^{n\times n}$ whose real part is positive definite, and flux and growth functions $f, g \in \cC^3(\C^n, \C^n)$.

This class of equation is known to admit a wide variety of pattern solutions: Constants, space-periodic solutions, homoclinic and heteroclinic profiles. Our interest lies in long time stability of some of these specific solutions.

\paragraph{Orbital stability}

For an hyperbolic space-constant equilibrium $\uu$, the dynamic is governed at leading order by the ordinary differential equation $u'(t) = g(u)$ and asymptotic stability is equivalent to the  spectral condition that the Jacobian matrix $J_g(\uu)$ is definite negative.
When the equilibrium $\uu$ is non-constant, the situation becomes more subtle. Since \eqref{e:main} is invariant under the transformation $u(t, \cdot) \mapsto u(t, \cdot + x_0)$, stability is only expected up to a spatial shift. 
From a technical point of view, taking the limit $x_0 \to 0$ guarantees that $0$ belongs to the linearized dynamic spectrum, with associated eigenmode $\uu_x$. Such an element of the spectrum is an obstruction to temporal decay of the linear dynamic.
A suitable way to handle this element is to absorb its contribution by modulating parameters. The wave $\uu$ is part of a manifold of equilibria described by physical parameters such as position, speed or wave number. By making them time and/or space dependent, the parameter dynamics can be used to match the non-decaying component corresponding to the $0$ eigenvalue. Depending on the background wave properties, this is done in different ways. 

\paragraph{Known phase shift constructions}
For homoclinic and heteroclinic connections to linearly stable constant states (often referred to as pulses and bistable fronts), $0$ is an isolated eigenvalue, with space-localized eigenmode. The standard spectral projections onto $\Span(\uu_x)$ allows to decouple the dynamic into two components. The first is one dimensional and corresponds to position deformation of the wave, the second is infinite dimensional, has spectral gap, and account for shape deformation of the wave. We refer to \cite{Sattinger-76} for a stability proof in this context.

In the case $g = 0$ of diffusive conservation law systems, homoclinic and heteroclinic connections to linearly neutral constant states can be found. Despite the fact that $\uu_x$ is still a space-localized eigenmode, spectral curves meet the imaginary axis at $0\in \C$, preventing the direct use of finite-dimensional projections. Separation of these two type of spectrum was possible through construction of an Evans function in \cite{Howard-Zumbrun-98}. Its zeros correspond to resolvent kernel poles, making possible the definition of projection through complex integration. Main difference with the previous case is that the shape dynamic has no spectral gap anymore. Nonlinear stability is obtained using polynomial weights, improving temporal decay.

For periodic wave solutions, the spectrum is essential, and a contact between the imaginary axis and a continuous curve occurs at $0$, similarly to the previous setting. However, the translation eigenmode $\uu_x$ is not space-localized, preventing any finite-dimensional spectral projections. In this setting, $0$ is one of the approximated eigenvalues of the essential curve, and can not be separated from nearby elements of the spectrum. A low-high frequency cut-off that relies on Floquet-Bloch multipliers is key to identify a shape deformation dynamic with enhanced algebraic decay and a phase deformation dynamic with a good non-linear structure. The use of polynomial weights is not possible here, and a sharp analysis is necessary to prove non-linear stability \cite{Johnson-Noble-Rodrigues-Zumbrun-13a}.

\paragraph{Unknown phase shift constructions}
In addition to the previously discussed waves, there are two other patterns worth mentioning. 
The first is an heteroclinic connection between two space-periodic patterns. As of today, both existence and stability are open questions, and we will not dwell on that topic. The recent study \cite{Avery-Carter-Rijk-Scheel-25} makes a first step in this direction.

The second is the family of heteroclinic waves between a linearly stable and linearly unstable constants, often referred to as monostable waves. Up to now, stability was only proven in weighted spaces that remove the zero eigenvalue \cite{Sattinger-76,Gallay-94,Faye-Holzer-18,Avery-Scheel-21}: Distance to the wave profile is measured with weighted norms that make the translational eigenmode $\uu_x/\varOmega$ unbounded as $x\to +\infty$. As a result, the class of initial perturbation is too small to allow for any phase dynamics, since an arbitrary small translation is infinitely far away from the background wave. 

Stability of monostable waves in optimal weighted topology, that is when $\uu_x/\varOmega$ is bounded, was recently obtained in \cite{Garenaux-Rodrigues-25} for scalar equations without diffusion. In this setting, a continuous region of essential spectrum touches the imaginary axis at $0$. Every point of this region is of multiplicity one, and the family of eigenmode is point-wise smooth with respect to the temporal spectral parameter.\footnote{Let $\phi_1^+(\lambda, x)$ be the family of eigenmode. For each $x\in \R$ and $\lambda_0$ in the spectrum, $\lambda \mapsto \phi_1^+(\lambda, x)$ is smooth from a neighborhood of $\lambda_0$.} In the interior of this region eigenmodes are space-localized, while at the border they have Fourier behavior $\re^{\ri \xi x}$ when $x \to +\infty$. Similarly to the periodic waves setting, the translational eigenmode $\uu_x / \varOmega$ can not be separated from nearby elements. 
Using transport from the unstable endstate to the wave, it was possible to convert spatial localization into temporal decay. This procedure was handled with arbitrary weak localization, and almost no temporal rate loss.

The goal of the present article is to extend the previous result to the diffusive system case \eqref{e:main}, see Theorem \ref{t:main}. Following the resolvent kernel approach from \cite{Howard-Zumbrun-98}, we identify a transported Gaussian kernel originating from low frequencies. It allows to recover temporal decay from arbitrarily weak spatial localization as highlighted in \cite{Garenaux-Rodrigues-25}, and thus to retain phase dynamics.

At the end of the article, we discuss in Theorem \ref{t:Fourier-Evans} the difference between Fourier and Evans eigenmode, as well as a criteria to separate them.

\subsection{Main statement}

We first assume that \eqref{e:main} admits a monostable wave connecting two constants states.

\begin{assumption}
\label{a:equilibrium}
(\textit{Background wave})
There exists $U^-, U^+ \in \C^n$ such that:
\begin{itemize}
\item \textit{Constant equilibria:} $g(U^\pm) = 0$, 
\item \textit{Stable and unstable equilibria:} The jacobian matrix $J_g(U^-)$ is definite negative, while the spectrum $\sigma(J_g(U^+))$ intersects $\Set{\lambda \in \C: \real{\lambda} > 0}$.
\item \textit{Traveling wave connection:} There exists a profile $\uu: \R \to \C^n$ and a speed $c>0$ such that 
\begin{equation}
\label{e:profile-eq}
- c \uu' = d \uu'' + (f(\uu))' + g(\uu)
\end{equation}
and $\lim_{\pm \infty} \uu = U^\pm$.
\item \textit{Generic localization:} The solution $\uu$ does not lie in the strong stable manifold of the equilibrium $U^+$ for the ordinary differential equation \eqref{e:profile-eq}.
\end{itemize}
\end{assumption}

Stability of the wave will be stated in term of a linear operator $\cA: \buc{2}(\R, \C^n) \to \buc{0}(\R, \C^n)$ 
\begin{equation*}
\cA = d \partial_{xx} + (J_f(\uu) + c)\partial_x + J_g(\uu) + (J_f(\uu))_x, 
\end{equation*}
obtained from linearization of \eqref{e:main} at $\uu$. The wave stability can be decomposed into three regions: The two infinities where the solution is close to constants $U^-$ and $U^+$, and the inner transition region. 
\begin{definition}[Dispersion relations]
\label{d:dispersion-relation}
For all $\nu \in \C$, let
\begin{equation*}
\cA^\pm[\nu] \deq d \nu^2 + (J_f(U^\pm) + c) \nu + J_g(U^\pm)
\end{equation*}
be the symbols associated to the asymptotic operators ($x\to \pm\infty$). For all $\lambda \in \C$, we will write
\begin{equation*}
d^\pm(\lambda, \nu) \deq \det\left(\lambda - \cA^\pm[\nu]\right), 
\end{equation*}
and will further refer to the equations $d^\pm(\lambda, \nu) = 0$ as the dispersion relations.
\end{definition}

From standard constant coefficient ordinary differential equations theory \cite{Kapitula-Promislow-13}, the $\buc{0}(\R, \C^n)$ spectra of $\cA^\pm$ are composed of $n$ parameterized curves (counting multiplicity)
\begin{equation*}
\sigma(\cA^\pm) = \Set{\lambda \in \C : \text{There exists } \xi \in \R \text{ such that } d^\pm(\lambda, \ri\xi) = 0}.
\end{equation*}
Some of these curves intersect the unstable half plane $\Set{\lambda \in \C: \real{\lambda} > 0}$, reflecting $U^+$ instability. It carries over to $\cA$ spectrum, and distance to the wave must therefore be measured in weighted space. We assume the existence of a scalar weight that stabilizes $\cA$ and retains the phase dynamic. 

For a $\kappa > 0$, let $\cL^+ : \buc{2}(\R) \to \buc{0}(\R)$ be the constant coefficient differential operator defined by
\begin{equation*}
\cL^+ v = \re^{\kappa x} \cA^+ (\re^{-\kappa x} v),
\end{equation*}
and let $\cL^- = \cA^-$. Again their respective spectra are composed of $n$ curves that we denote 
\begin{equation*}
\xi \mapsto \lambda_k^\pm(\xi)
\hspace{4em}
\xi \in \R, 
\hspace{2em}
1 \leq k \leq n.
\end{equation*}
For completeness, let us note that the $\lambda^+$ curves are the solutions to 
\begin{equation*}
d^+(\lambda^+, \ri \xi - \kappa) = \det(\lambda - \cL^+[\ri \xi]) = 0,
\hspace{4em} \xi \in \R.
\end{equation*}

\begin{assumption}
\label{a:essential-spectrum}
(\textit{Essential spectrum})
There exists $\kappa > 0$ and $\xi_0 \in \R$ such that the following holds. There exists positive constants $\varXi$, $\eta$, $\alpha_1$ and a complex $\alpha_2$ with $\real{\alpha_2} > 0$ such that:
\begin{itemize}
\item \textit{Single marginal point:} When $\absolute{\xi - \xi_0} \leq \varXi$, 
\begin{equation}
\label{e:marginal-lambda-expansion}
\lambda_1^+(\xi) = \ri \alpha_1 (\xi - \xi_0) - \alpha_2 (\xi - \xi_0)^2 + \cO_{\xi \to \xi_0}((\xi - \xi_0)^3).
\end{equation}
\item \textit{Spectral gap 1:} When $\absolute{\xi - \xi_0} \geq \varXi$,
\begin{equation*}
\real{\lambda_1^+(\xi)} \leq - \eta.
\end{equation*}
\item \textit{Spectral gap 2:} For all $\xi \in \R$ and all $1 \leq k \leq n$
\begin{align*}
& \real{\lambda_k^-(\xi)} \leq - \eta, \\
k\neq 1 \ \implies \ & \real{\lambda_k^+(\xi)} \leq - \eta.
\end{align*}
\item \textit{Optimal weight:} $x\mapsto \re^{\kappa x} \uu'(x) \in L^\infty([0,+\infty))$.
\end{itemize}
\end{assumption}

To each solution $(\lambda, \nu)$ of the dispersion relation $d^+ = 0$ is associated a projection onto $\ker \lambda - \cL^+[\nu]$ and parallel to $\range \lambda - \cL^+[\nu]$. We denote $\pi$ the projection at $(\lambda, \nu) = (0, \ri \xi_0)$.

In addition to essential spectrum, $\sigma(\cA)$ also contains point spectrum, which corresponds to zeros of the Evans function. We will follow a standard argument to construct this holomorphic function in section \ref{ss:ODE-solution} (see page \pageref{proof:spatial-eigenvalues-2}), and for now assume that the point spectrum is stable. Given $(\kappa, \xi_0)$ satisfying Assumption \ref{a:essential-spectrum}, we need a weight with exponential behavior at $\pm \infty$. 

\begin{assumption}
\label{a:point-spectrum}
(\textit{Point spectrum})
There exists $\theta>0$ and $\varOmega \in \buc{2}(\R, \C)$ such that 
\begin{itemize}
\item \textit{Exponential behavior:} For all $x\in \R$, $\absolute{\varOmega(x)} \in (0,1]$, and 
\begin{equation*}
\varOmega(x) = 
\begin{cases}
\re^{(-\kappa + \ri \xi_0) x} & x \geq 1, \\
1 & x \leq -1.
\end{cases}
\end{equation*}
\item \textit{Stable point spectrum:} Let $\cL: \buc{2}(\R, \C^n) \to \buc{0}(\R, \C^n)$ be the differential operator 
\begin{equation*}
\cL v \deq \frac{1}{\varOmega} \cA (\varOmega v),
\end{equation*}
then the Evans function associated to $\cL$ does not vanish on
\begin{equation*}
\label{e:def-Lambda-theta}
\varLambda = \Set{\lambda \in \C : \real{\lambda} \geq - \theta (1 + \absolute{\imag{\lambda}})}.
\end{equation*} 
\item \textit{Spatial Jordan block} The set 
\begin{equation}
\label{e:non-Jordan-set}
\Set{\lambda \in \varLambda : \text{ the matrices } A^\pm(\lambda) \text{defined in \eqref{e:1st-order-pb} are diagonalizable}}
\end{equation}
has no accumulation point, and does not contains $0$. \\ 
Remark that $\lambda$ belongs to \eqref{e:non-Jordan-set} if and only if every solution $(\nu, v_1, v_2) \in \C \times \C^n \times C^n$ to 
\begin{equation*}
\begin{cases}
(\lambda - \cL^\pm[\nu]) v_1 = 0, \\
(\lambda - \cL^\pm[\nu]) v_2 = -\partial_\nu \cL^\pm[\nu] v_1.
\end{cases}
\end{equation*}
satisfies $v_1 = v_2 = 0$.
\end{itemize}
\end{assumption}

To describe phase dynamics, we will need two additional definitions. To measure spatial localization, we will compose $L^\infty$ norms with a weight family. We need it to be sharp enough that it captures arbitrarily weak spatial localization. 
\begin{definition}
\label{d:sub-exp}
Let $\rho : \R \to [1,+\infty)$ be non-decreasing, and let $\eta>0$, $M\geq 1$. We say that $\rho$ is a sub-exponential weight with constants $(\eta, M)$ if for all $x\geq 0$
\begin{equation*}
\rho(x) \leq M \re^{\eta x},
\end{equation*}
and 
\begin{equation*}
\int_0^x \re^{-\eta (x - y)} \frac{1}{\rho(y)} \d y \leq \frac{M}{\rho(x)}.
\end{equation*}
\end{definition}
We provide in Lemma \ref{l:exemple-sub-exp-weights} examples of such weights. At last, we remark that Assumption \ref{a:essential-spectrum} guarantees the kernel of the matrix $\lambda - \cL[\ri \xi]$ is one dimensional at $(\lambda, \xi) = (0, \xi_0)$. We denote $\pi$ the spectral projection onto this kernel. We are now ready to state our main result.

\begin{theorem}
\label{t:main}
Assume that Assumptions \ref{a:equilibrium}, \ref{a:essential-spectrum} and \ref{a:point-spectrum} hold, and let $M\geq 1$. There exists positive constants $C$ and $\delta$ such that the following three points hold.
\begin{enumerate}
\item \label{i:th-1} If $u_0 \in \buc{3}(\R)$ satisfies 
\begin{equation*}
E_0 \deq \Norm{\frac{u_0 - \uu}{\varOmega}}_{W^{2,\infty}(\R)} \leq \delta
\end{equation*}
then there exists a unique solution 
\begin{equation*}
u \in \cC^0([0,+\infty), \buc{3}(\R)) \ \cap \ \cC^1((0,+\infty), \buc{1}(\R))
\end{equation*}
to \eqref{e:main} with initial data $u(0) = u_0$. It satisfies
\begin{equation*}
\label{e:th-no-decay}
\Norm{\frac{u(t, \cdot + ct) - \uu}{\varOmega}}_{W^{1,\infty}(\R)} \leq C E_0.
\end{equation*}
\item \label{i:th-2} If in addition $\pi \frac{u_0 - \uu}{\varOmega}$ converges to a finite limit at $x\to +\infty$, then there exists $x_\infty \in \R$ such that  $\absolute{x_\infty} \leq C E_0$ and 
\begin{equation*}
\label{e:th-rateless-decay}
\Norm{\frac{u(t, \cdot + ct + x_\infty) - \uu}{\varOmega}}_{W^{1,\infty}(\R)} \underset{t\to +\infty}{\longrightarrow} 0
\end{equation*}
when $t\to +\infty$.
\item \label{i:th-3} If in addition to the previous points, there exists a sub-exponential weight $\rho$ with constants $(\delta, M)$ satisfying $\rho(x) \leq \sqrt{1 + x}$ and such that
\begin{equation*}
E_{0, \rho} \deq \Norm{\pi \frac{u_0(\cdot + x_\infty) - \uu}{\varOmega} \rho}_{W^{1,\infty}(\R)}
\end{equation*}
is finite, then for all $t\geq 0$
\begin{equation*}
\label{e:th-rate-decay}
\Norm{\frac{u(t, \cdot + ct + x_\infty) - \uu}{\varOmega}}_{W^{1,\infty}(\R)} \leq C \frac{E_{0,\rho}}{\rho(C t)}.
\end{equation*}
\end{enumerate}
\end{theorem}

\begin{lemma}
\label{l:exemple-sub-exp-weights}
For all $\eta>0$ and $M\geq 1$, there exists $a_0$ such that for all $a \in (0, a_0)$, the maps $x\mapsto \max(1, x)^a$, $x\mapsto \ln\big(\re + \max(0, x)\big)^a$ and $x\mapsto 1$ are sub-exponential weights with constants $(\eta, M)$.
\end{lemma}
\begin{proof}
We begin with the following remark. Let $\tilde{\eta}>0$ and $\eta> \tilde{\eta}$. If a non-decreasing $\rho : \R \to [1,+\infty)$ satisfies $\rho(0) = 1$ and for all $0 \leq y \leq x$
\begin{equation}
\label{e:sub-exp-alternative-estimate}
\frac{\rho(x)}{\rho(y)} \leq M \re^{\tilde{\eta}(x - y)},
\end{equation}
then $\rho$ is a sub-exponential weight with constants $(\eta, M)$. 
Therefore, it is sufficient to prove that the three mentioned examples satisfy \eqref{e:sub-exp-alternative-estimate}. To do so, one may remark that there exists an $a$-dependent $x_0>0$ such that the map $R(x) = \rho(x) \re^{-\tilde{\eta} x}$ is decreasing on $[x_0, +\infty)$, providing \eqref{e:sub-exp-alternative-estimate} when $y \in [x_0, x]$ with $M = 1$. When $y \in [0, x_0]$, it is direct to get the successive estimates 
\begin{equation*}
R(x) \leq R(x_0) \leq \frac{R(x_0)}{\inf_{[0, x_0]} R} R(y) \leq \rho(x_0) R(y).
\end{equation*}
Computation of $\rho(x_0)$ provides the constant $a_0$.
\end{proof}

\begin{remark}
Assumption \ref{a:essential-spectrum} implies that $\uu'$ does not contribute to a zero of the Evans function, and that the later does not vanish at the origin.

If item \textit{Generic localization} of Assumption \ref{a:equilibrium} does not holds, then the profile derivative contribute to a zero of the Evans function, and it is possible to separate this Evans mode from the rest of the spectrum with a strong enough weight. Such a setting is closer to that of bistable waves or of diffusive conservation laws, for which it is known how to isolate phase dynamic. The \textit{Optimal weight} item of Assumption \ref{a:essential-spectrum} is not sharp in that case, since the weighted eigenmode converges to $0$ at $+\infty$.
\end{remark}

\subsection{Discussion}

\begin{remark}
To prove orbital stability of travelling waves when the essential spectrum touches the imaginary axis, \cite{Johnson-Noble-Rodrigues-Zumbrun-13a} used a different nonlinear scheme. With a similar decomposition of the resolvent kernel, they decouple the nonlinear dynamics into a phase and a shape contribution. One of the interesting feature of this approach was that it allowed to describe the evolution of parameters through the reduction to a Whitham system \cite{Johnson-Noble-Rodrigues-Zumbrun-13b}. This approach would be interesting to treat the case of slowest reaction-diffusion monostable fronts. Unfortunately, it does not seem to apply in the current setting of weighted dynamics. Indeed, conjugation by exponential weight usually destroys the derivative structure of some nonlinearities, weakening their temporal decay and preventing to close the nonlinear argument. In particular, a weighted correction ansatz of the form
\begin{equation*}
u(t, x) = \uu(x - ct - \psi(t)) + \varOmega(x - ct - \psi(t)) v(t, x - ct - \psi(t))
\end{equation*}
creates a term 
\begin{equation*}
\frac{1}{\varOmega} \, \partial_t \psi \, \partial_x (\varOmega v)
\end{equation*}
that can not be written as a full derivative. The adapted ansatz
\begin{equation*}
u(t, x) = \uu(x - ct - \psi(t)) + \varOmega(x - ct) v(t, x - ct)
\end{equation*}
would solve this issue, but leads to technical difficulties, since one would need to prove resolvent kernel estimates for time-dependent linear operators.
\end{remark}

\subsection{Future directions}

\paragraph{Matrix valued weights}
We restrict here to scalar valued weight $\varOmega$, which modify all essential curves $\lambda_k^\pm$ at once. It would be interesting to investigate whether matrix valued weights allow to stabilize only some of these curves.

\paragraph{Logarithmic shifts}
An important part of the litterature deals with close to constants initial data. If the solution at initial time is constant outside a compact interval, stability results to the critical wave, with a logarithmic shift are known, see the recent \cite{Avery-Scheel-22} and references therein. This class of result is slightly distinct from what we are presenting here, since they handle initial data which is not close from the wave profile: in optimal weighted topology, a constant by part function is $\mathcal{O}(1)$-close to $\uu$.

\paragraph{Critical speed}
The eigenvalue expansion from assumption \ref{a:essential-spectrum} guarantees non-zero convection for the less stable mode. It reflects the fact that the considered wave is supercritical, meaning that its speed is strictly above the minimal one. For critical waves, advection vanishes at the unstable infinity, and the approach we present here fails. 
The linear dynamic in optimal weighted spaces can be depicted by the scalar toy model 
\begin{equation*}
u_t = \frac{1}{x} u_x,
\hspace{4em}
x \geq 1,
\end{equation*}
whose solution is given through characteristics method by
\begin{equation*}
u(t,x) = u(0, \sqrt{2t + x^2}),
\end{equation*}
leading to 
\begin{equation*}
\norm{u(t,\cdot)}_{L^\infty([1,+\infty))} \leq \frac{1}{\rho(\sqrt{2t})} \norm{\rho u(0, \cdot)}_{L^\infty([1,+\infty))}.
\end{equation*}
However, it is not clear how to recover similar transport behavior in the diffusive setting.

\paragraph{Orbital stability with respect to speed}
In known examples \cite{Kolmogorov-Petrovsky-Piskunov-37,vandenBerg-Hulshof-Vandervorst-01,Avery-Garenaux-23}, families of monostable waves are parameterized by phase shift and speed. Orbital stability with respect to speed is known for Korteweg-de-Vries solitons \cite{Pego-Weinstein-94}, not for monostable waves. A major obstacle to solve this problem is the fact that waves with different speeds come with different optimal topologies (that is, the localization rate $\kappa$ in Assumption \ref{a:essential-spectrum} is speed dependent). As a consequence, the difference between two nearby waves may be exponentially unbounded in optimal topology.

\section{Notations}

During the nonlinear study, we will use the notation $T_{2, g}(a, h)$ to stand for the second order Taylor expansion of a map $g:\C^n \to \C^n$ at $a \in \C^n$ in the direction $h \in \C^n$:
\begin{equation}
\label{e:Taylor-expansion}
T_{2, g}(a, h) = g(a + h) - g(a) - J_{g}(a) h.
\end{equation}

We will write $\sqrt{\cdot}$ to stand for the complex square root with branch cut at $(-\infty, 0)$. For all $\lambda \in \C\backslash (-\infty, 0)$, $\sqrt{\lambda} \in \Set{\real{\lambda} > 0}$.

\begin{lemma}
\label{l:front}
Let $\kappa>0$ and $\xi_0\in \R$ from Assumption \ref{a:essential-spectrum}. There exists $V\in \C^{2n \times 1}$, $\varepsilon>0$ and $C>0$ such that for every $x\geq 0$
\begin{align*}
\Absolute{\begin{pmatrix}
\uu - U_+ \\ 
\uu'
\end{pmatrix}
 - 
V \re^{(-\kappa + \ri \xi_0) x}}
\leq C \re^{-(\kappa + \varepsilon) x}.
\end{align*}
\end{lemma}
\begin{proof}
Rewriting the profile equation as a first order ordinary differential equation, and linearizing near $U_\pm$, we obtain for the unknown $\textbf{u} = \transpose{\begin{pmatrix}
\uu - U_+ & \uu'
\end{pmatrix}}$
\begin{equation*}
\textbf{u} = A \textbf{u} + N(\textbf{u}).
\end{equation*}
with linearized matrix 
\begin{equation*}
A = \begin{pmatrix}
0 & \id[n\times n] \\
d^{-1} g'(U_\pm) & d^{-1} f'(U_\pm)
\end{pmatrix}
\end{equation*} and nonlinear terms $N(\textbf{u}) = \mathcal{O}(\textbf{u}^2)$. A fixed point argument on the associated Duhamel formulation ensures the existence of a vector $V(x)$ with polynomial $x$-dependence and an eigenvalue $\tau \in \C$ to $A$ such that 
\begin{equation*}
\absolute{\mathbf{u}(x) - P(x) \re^{\tau x}} \leq C \re^{(\real{\tau} - \varepsilon) x}
\end{equation*}
when $x\to +\infty$. Using Assumption \ref{a:essential-spectrum} item \textit{optimal weight}, we conclude that $\real{\tau} \leq -\kappa$. 
Assumption \ref{a:essential-spectrum} item \textit{single marginal point} ensures that $-\kappa + \ri \xi_0$ is also an eigenvalue for $A$ and Assumption \ref{a:equilibrium} item \textit{Generic localization} guarantees $\real{\tau} \geq -\kappa$. 

As a consequence, $\real{\tau} = -\kappa$ and from Assumption \ref{a:essential-spectrum} item \textit{spectral gap} we deduce that $\tau = -\kappa + \ri \xi_0$. Applying again Assumption \ref{a:essential-spectrum} item \textit{optimal weight} we deduce that $V$ is constant. The proof is complete.
\end{proof}

\begin{remark}
An immediate consequence of the previous Lemma is that $\Omega = \absolute{\uu - U_+}$ satisfies Assumption \ref{a:point-spectrum}, item \textit{Exponential behavior}. It is necessary and sufficient to check Assumption \ref{a:point-spectrum} for this choice of $\Omega$, because the spectrum of $\cL$ is unchanged if one replaces $\Omega$ with an equivalent weight
\begin{equation*}
\frac{1}{C} \varOmega \leq \tilde{\varOmega} \leq C \varOmega.
\end{equation*}

Another consequence is that $\cA$ and $\cL$ coefficients exponentially converge to their limits. In the following we write 
\begin{equation*}
\cL(x) v \deq \ell_2 v_{xx} + \ell_1(x) v_x + \ell_0 v.
\end{equation*}
\end{remark}

\section{Resolvent estimates}
To estimate the resolvent, we will use a kernel formulation. To do so, we need to solve the resolvent problem
\begin{equation}
\label{e:res-pb}
\lambda \phi - \cL(x) \phi = \delta_y,
\hspace{4em}
\phi : \R\backslash\Set{y} \lra \C^{n\times 1},
\end{equation}
together with the transpose resolvent problem
\begin{equation}
\label{e:trans-res-pb}
\lambda \dual{\phi} - \transpose{\cL}(y) \dual{\phi} = \delta_x,
\hspace{4em}
\dual{\phi} : \R\backslash\Set{x} \lra \C^{1\times n},
\end{equation}
where $\delta$ is the Dirac delta distribution on $\R$, and the transpose operator is defined by 
\begin{equation*}
\label{e:def-transpose-A}
\transpose{\cL} \dual{\phi} (y) \deq \partial_{yy}(\dual{\phi} \, \ell_2) - \partial_y (\dual{\phi} \, \ell_1(y)) + \dual{\phi} \, \ell_0(y).
\end{equation*}
In section \ref{ss:ODE-solution}, we introduce some notations and solves the ODE associated with \eqref{e:res-pb} and \eqref{e:trans-res-pb}. In section \ref{ss:construction-resolvent-kernel} we then use these solutions to construct the resolvent kernel and decompose it in a suitable way. In section \ref{ss:bound-resolvent-kernel} we estimate its different parts. We advise to start with the reading of section \ref{ss:construction-resolvent-kernel}.

\subsection{Solutions of the eigenproblems}
\label{ss:ODE-solution}
We start by solving the eigenvalue problem associated with \eqref{e:res-pb}:
\begin{equation}
\label{e:pvp}
\lambda \phi - \cL(x) \phi = 0.
\end{equation}
It will be convenient to recast this system of second order ODE to a system of first order ODE. Let $\Phi$ be the element of $\C^{2n \times 1}$ obtained by concatenation of the vectors $\phi$ and $\partial_x \phi$. It satisfies
\begin{equation}
\label{e:1st-order-pb}
\partial_x \Phi(\lambda, x) = A(\lambda, x) \, \Phi(\lambda, x),
\hspace{4em}
A(\lambda, x) \deq 
\begin{pmatrix}
\zero[n\times n] & \id[n\times n]\\
\ell_2^{-1}(\lambda - \ell_0(x)) & - \ell_2^{-1} \ell_1(x)
\end{pmatrix}.
\end{equation}
Here and in the following, $\id[n\times n]$ stands for the $n\times n$ identity matrix.
We proceed similarly for the transpose eigenvalue problem:
$\dual{\Phi} \deq \begin{pmatrix} \dual{\phi} & \partial_y \dual{\phi}\end{pmatrix}$ satisfies
\begin{equation}
\label{e:1st-order-trans-pb}
\partial_x \dual{\Phi}(\lambda, y) = \dual{\Phi}(\lambda, y) \, B(\lambda, y),
\hspace{4em}
B(\lambda, y) \deq 
\begin{pmatrix}
\zero[n\times n] & (\lambda - \ell_0(y)+\ell_1'(y)) \, \ell_2^{-1}\\
\id[n\times n] & \ell_1(y) \, \ell_2^{-1}
\end{pmatrix}.
\end{equation}
In this section, we first define and estimate spatial eigenvalues. Then, we construct holomorphic bases of solutions for the non-autonomous ODE \eqref{e:1st-order-pb} and \eqref{e:1st-order-trans-pb}, in a dual way.

\subsubsection{Spatial eigenvalues}
As we will see in the following, solutions to \eqref{e:1st-order-pb} are well approximated when $x\to +\infty$ by solutions of the asymptotic problem
\begin{equation}
\label{e:asymptotic-1st-order-pb}
\partial_x \Phi = A^+(\lambda) \Phi, 
\end{equation}
where $A^+(\lambda) \deq \lim_{x\to +\infty} A(\lambda, x)$. Because \eqref{e:asymptotic-1st-order-pb} is a linear ODE with constant coefficients, its solution are the $x\mapsto \re^{A^+(\lambda) x} V$, for any given $V\in \C^{n\times 1}$. In particular, their behavior is directly linked to the eigenvalues of $A^+(\lambda)$. In the following, we will refer to them as the spatial eigenvalues, and label them 
\begin{equation*}
\nup{-n}(\lambda), \dots, \nup{-1}(\lambda), \nup{1}(\lambda) \dots, \nup{n}(\lambda).
\end{equation*}
Using the companion structure of $A^+(\lambda)$, it is not difficult to see that $(\lambda, \nu^+(\lambda))$ solves the dispersion relation $d^+(\lambda, \nu) = 0$, see Definition \ref{d:dispersion-relation}.
We will use similar notations for solutions of the left asymptotic problem: $A^-(\lambda) \deq \lim_{x\to -\infty} A(\lambda, x)$ has eigenvalues $\num{-n}(\lambda), \dots, \num{-1}(\lambda), \num{1}(\lambda) \dots, \num{n}(\lambda)$. Up to relabeling, we can assume that when $\lambda \in \R$ lies outside of $\spectrum[ess](\cL)$, they are ordered by real part
\begin{equation*}
\real{\nupm{-n}(\lambda)} \leq \dots \leq \real{\nupm{-1}(\lambda)} \leq 0 \leq \real{\nupm{1}(\lambda)} \leq \dots \leq \real{\nupm{n}(\lambda)}.
\end{equation*}
Here and in the following, the symbols $\pm$ and $\mp$ can be replaced in the whole expression, either by $+$ and $-$, or by $-$ and $+$. That is, the above inequality implies no ordering between $\real{\nup{1}(\lambda)}$ and $\real{\num{2}(\lambda)}$. 

This labeling is extended continuously with respect to $\lambda \in \varLambda$. We further group the spatial eigenvalues into three sets:
\begin{equation*}
N_s \deq \Set{\nupm{-n}, \dots, \nupm{-1}}, 
\hspace{4em}
\Set{\nup{1}},
\hspace{4em}
N_u \deq \Set{\num{1}, \nupm{2}, \dots, \nupm{n}}.
\end{equation*}
\begin{lemma}
\label{l:spatial-eigenvalues}
Let $m>0$. Assuming $\theta>0$ is small enough (see Assumption \ref{a:point-spectrum}), there exists $C>0$ such that, with $r:\varLambda \to [1,+\infty)$ defined by
\begin{equation}
\label{e:def-rate}
r(\lambda) \deq C (1 + \absolute{\lambda}^{\frac12}),
\end{equation}
the following two points holds.
\begin{enumerate}[label=(\roman*)]
\item For all $\lambda \in \varLambda$, for all $\nu_s \in N_s$ and $\nu_u \in N_u$:
\begin{equation}
\label{e:spatial-gap}
\real{\nu_s(\lambda)} \leq - r(\lambda), 
\hspace{4em}
r(\lambda) \leq \real{\nu_u(\lambda)}.
\end{equation}
\item \label{i:spatial-gap-2} 
$\nup{1}(0)$ is a simple eigenvalue of $A^+(0)$, and it satisfies
\begin{equation*}
\real{\nup{1}(0)} = 0, 
\hspace{4em}
\partial_\lambda \nup{1}(0) \in (0,+\infty).
\end{equation*}
If $\lambda \in \varLambda \backslash B(0,m)$, then
\begin{equation}
\label{e:spectral-gap-2}
r(\lambda) \leq \real{\nup{1}(\lambda)}.
\end{equation}
\end{enumerate}
\end{lemma}
\begin{proof}[Proof -- part 1]
We start by proving \eqref{e:spectral-gap-2} and \eqref{e:spectral-gap-2} when $\absolute{\lambda} \geq M$ for a large enough $M$. The proof when $\absolute{\lambda} \leq M$ relies on a yet undefined notation, and is done later, see page \pageref{proof:spatial-eigenvalues-2}.

As mentioned above, $\nupm{i}(\lambda)$ solve the dispersion relations
\begin{equation*}
\label{e:dispersion-relation}
\det(\lambda - \cL^\pm[\nu]) = 0,
\end{equation*}
and we now solve these. With the change of variable $z = \nu \absolute{\lambda}^{-\frac12}$, it rewrites
\begin{equation*}
g(z) \deq \det \left(\frac{\lambda}{\absolute{\lambda}} - \frac{1}{\absolute{\lambda}} \cL^\pm\left[z \absolute{\lambda}^{\frac12}\right] \right) = 0.
\end{equation*}
Let us introduce $f(z) \deq \det\left( \frac{\lambda}{\absolute{\lambda}} - \ell_2 z^2\right)$, and remark that the zeros of $f$ are the $\pm \sqrt{\frac{\lambda}{d_j\absolute{\lambda}}}$, $1\leq j \leq n$. Let $\varGamma$ be a closed path surrounding one of these root (with multiplicity eventually larger than $1$) but not the other ones. Then there exists $\varepsilon>0$ such that $\absolute{f(z)} > \varepsilon$ for all $z\in \varGamma$. On the other hand, there exists $C>0$ such that for $z \in \varGamma$
\begin{equation*}
\absolute{f(z) - g(z)} \leq C (\absolute{z} \absolute{\lambda}^{-\frac12} + \absolute{\lambda}^{-1}) \leq \varepsilon,
\end{equation*}
where the last inequality is obtained by taking $\absolute{\lambda}$ large enough (remark that zeros of $f$, and thus $\varGamma$, do not depend on $\absolute{\lambda}$). By applying Rouch{\'e}'s Theorem, we conclude that $f$ and $g$ have the same number of zeros in the interior of $\varGamma$.
Because all $d_j$ are positive, and because $\absolute{\arg\sqrt{\lambda}} < \frac\pi2$ when $\lambda \in \varLambda$, we can choose paths $\varGamma$ avoiding the imaginary axis $\ri \R$. Thus the zeros of $g$ are away from $\ri \R$, uniformly with respect to $\lambda$ outside of a large enough ball. Reverting the change of variable, we get that $\absolute{\real{\nupm{j}(\lambda)}} \geq C \absolute{\lambda}^{\frac12}$ for such values of $\lambda$. The first part of the proof is complete.
\end{proof}

\subsubsection{Holomorphic bases of solutions with asymptotic behavior} Our goal here is to construct solutions to \eqref{e:1st-order-pb} that are holomorphic with respect to $\lambda$. Let us describe the situation in the simpler case \eqref{e:asymptotic-1st-order-pb}.

On the one hand, the asymptotic behavior of solutions to \eqref{e:asymptotic-1st-order-pb} is clear by decomposing $\C^{2n \times 1}$ into generalized eigenspaces of $A^\pm$. In the following, for a map $\Phi:\R \to \C^{2n\times1}$, we denote $\eta^+(\Phi)$ its growth rate at $+\infty$, that is 
\begin{equation*}
\eta^+(\Phi) \deq \inf \Set{\eta\in \R : x\mapsto \re^{-\eta x} \Phi(x) \text{ is bounded over } \R_+},
\end{equation*}
together with a similar notation for the behavior at $-\infty$:
\begin{equation*}
\eta^-(\Phi) \deq \sup \Set{\eta\in \R : x\mapsto \re^{-\eta x} \Phi(x) \text{ is bounded over } \R_-}.
\end{equation*}

On the other hand, a typical issue is that several spatial eigenvalues $\nupm{}(\lambda)$ can collide at a given value  $\lambda_0$ of the parameter, potentially leading to a lack of holomorphicity of eigenvectors with respect to $\lambda$. See \cite[pages 64 and 72]{Kato-95} for simple examples. However, thanks to eigenspaces being invariant through $\re^{A^\pm(\lambda) x}$, holomorphicity of solutions to \eqref{e:asymptotic-1st-order-pb} reduces to holomorphicity of spectral projections of $A^\pm$. The latter is well known by perturbative approach \cite[Chapter two, \S 1]{Kato-95}, since $A^\pm$ are finite dimensional operators.

Although the previous argument breaks down when studying the non-autonomous ODE \eqref{e:1st-order-pb}, its solutions are well approximated by those of \eqref{e:asymptotic-1st-order-pb}, and holomorphicity follows from a fixed point procedure we now recall. We adapt the discussion \cite[section 3]{Howard-Zumbrun-98}, and in particular the Proposition 3.1 there, to be able to handle collision points where Jordan chains may appear.

\begin{lemma}
\label{l:ODE-behavior}
Let $\kappa > 0$ be given by Lemma \ref{l:front}.
Let $\lambda_0 \in \varLambda$, let $\eta \in \R$ and $\varepsilon>0$ small. Suppose that $V_1(\lambda), \dots, V_p(\lambda)$ are linearly independent and holomorphic elements of $\C^{2n\times 1}$, such that $\eta_j^\pm(\lambda) = \eta^\pm(x\mapsto \re^{A^\pm(\lambda) x}V_j(\lambda))$ satisfy
\begin{equation*}
\absolute{\eta_j^\pm(\lambda) - \eta}< \varepsilon,
\end{equation*}
for all $\lambda$ close to $\lambda_0$ and $1\leq j \leq p$.
Then, for all $1\leq j \leq p$, there exists a solution $\Phi(\lambda, x)$ of \eqref{e:1st-order-pb}, satisfying for all $x \in \R_\pm$ and locally around $\lambda_0$
\begin{equation}
\label{e:ODE-behavior}
\Phi(\lambda, x) = \re^{A^\pm(\lambda) x} V_j(\lambda) + \Cal{O}_{\lambda, x}(\re^{(\eta \mp \frac{\kappa}{2}) x}).
\end{equation}
The notation $\Cal{O}_{\lambda, x}(1)$ stands for a function which is holomorphic from a small neighborhood of $\lambda_0$ to $\buc{0}(\R_\pm)$.
\end{lemma}
\begin{proof}
We mostly follow \cite{Howard-Zumbrun-98}. To lighten notations, we only prove the $+$ case. The $-$ case is similar. From the change of variable $\Phi = \re^{(\eta + \varepsilon) x} F$, we see that $F(\lambda, x)$ satisfies
\begin{align}
\nonumber \partial_x F & {} = (A(\lambda, x) - \eta - \varepsilon) F, \\
\nonumber & {} = (A^+(\lambda) - \eta - \varepsilon) F + (A(\lambda, x) - A^+(\lambda)) F, \\
\label{e:fixed-point-EDO} & {} \deq A_{\eta+\varepsilon}(\lambda) F + R(x) F.
\end{align}
Remark that $R(x) \deq A(\lambda, x) - A^+(\lambda, x) = \Cal{O}(\re^{-\kappa x})$ does not depend on $\lambda$, and let $\kappa_1, \kappa_2 \in \R$ such that 
\begin{equation*}
-\kappa < -\kappa_1 < -\kappa_2 < -\frac{\kappa}{2} - \varepsilon.
\end{equation*}
We then let $P(\lambda)$ be the spectral projection corresponding to all eigenvalues $\nu$ of $A^+(\lambda)$ that satisfy
\begin{equation*}
\real{\nu} < \eta + \varepsilon - \kappa_2.
\end{equation*}
Remark that $1-P(\lambda)$ is a spectral projection corresponding to eigenvalues $\nu$ that satisfy
\begin{equation*}
\eta + \varepsilon - \kappa_1 < \eta + \varepsilon - \kappa_2 \leq  \real{\nu}.
\end{equation*}
In particular,
\begin{align*}
\absolute{\re^{A_{\eta+\varepsilon}(\lambda) x} P(\lambda)} \leq C\re^{-\kappa_2 x},
\hspace{4em}
x \geq 0,\\
\absolute{\re^{B(\lambda) x} (1-P(\lambda))} \leq C\re^{-\kappa_1 x},
\hspace{4em}
x \leq 0.
\end{align*}
For a fixed $1 \leq j \leq p$ and $x\geq 0$, we define
\begin{align*}
(\Cal{T} F)(\lambda, x) \deq {} & \re^{A_{\eta+\varepsilon}(\lambda) x} V_j(\lambda) - \int_x^{+\infty} \re^{A_{\eta+\varepsilon}(\lambda)(x-y)} P(\lambda) R(y) F(y) \d y \\
& {} + \int_M^x \re^{A_{\eta+\varepsilon}(\lambda)(x-y)} (1-P(\lambda)) R(y) F(y) \d y.
\end{align*}
Then $\partial_x (\Cal{T} F) = A_{\eta+\varepsilon}(\lambda) \Cal{T} F + R(x) F$, so that a fixed point of $\Cal{T}$ is a solution to \eqref{e:fixed-point-EDO}. Let $F_1, F_2 \in \buc{0}([M, +\infty))$ for a constant $M>0$, then
\begin{align}
\nonumber \absolute{\Cal{T}F_1 - \Cal{T}F_2}(\lambda, x) \leq {}& C \norm{F_1 - F_2}_{L^\infty} \int_x^{+\infty} \re^{-\kappa_1(x-y)} \re^{-\kappa y} \d y \\
\nonumber & {} + C\norm{F_1 - F_2}_{L^\infty} \int_M^x \re^{-\kappa_2(x-y)} \re^{-\kappa y} \d y,\\
\label{e:contraction-1} \leq {}& C \norm{F_1 - F_2}_{L^\infty} \left(\frac{\re^{-\kappa x}}{\kappa - \kappa_1} + \frac{\re^{-\kappa M -\kappa_2(x-M)}}{\kappa - \kappa_2} - \frac{\re^{-\kappa x}}{\kappa - \kappa_2}\right),\\
\label{e:contraction-2} \leq {}& C \norm{F_1 - F_2}_{L^\infty} \, \re^{-\kappa M}.
\end{align}
Due to the choice of $V_j$, $\re^{A_{\eta+\varepsilon}(\lambda) x} V_j(\lambda)$ is bounded on $[M,+\infty)$. Thus, \eqref{e:contraction-2} shows that $\Cal{T}$ is a contraction on $\buc{0}(M,\infty)$, and hence admits a fixed point $F_j(\lambda, x)$, locally holomorphic in  $\lambda$. In particular \eqref{e:contraction-1} ensures that
\begin{equation*}
F_j(\lambda, x) - \re^{A_{\eta+\varepsilon}(\lambda) x} V_j(\lambda) = \Cal{T} F_j(\lambda, x) - \Cal{T}(0)(\lambda, x) = \Cal{O}_{\lambda, x}( \norm{F_j(\lambda)}_{L^\infty} \, \re^{-\kappa_2 x}).
\end{equation*}
Coming back to the original variable $\Phi$, and using that $\eta + \varepsilon - \kappa_2 \leq \eta - \frac{\kappa}2$, we obtain the claimed bound on $\Phi(\lambda, x)- \re^{A^+(\lambda) x} V_j(\lambda)$.
\end{proof}

We now apply Lemma \ref{l:ODE-behavior} locally to all $\lambda_0 \in \varLambda$, as follow. 
For spatial eigenvalues $\nupm{j_0}$ with constant multiplicity, we can chose $\varepsilon$ so small that no other eigenvalue belongs to the ball $B(\nupm{j_0}(\lambda_0), \varepsilon)$. The associated generalized eigenvectors $(V_1(\lambda), \dots, V_p(\lambda))$ are then holomorphic.

At special points $\lambda_0$ where several spatial eigenvalues of $A^\pm$ collide, we let $\nu_0\in \C$ and $\varepsilon>0$ such that $B(\nu_0, \varepsilon)$ contains exactly the group of $p$ eigenvalues colliding, and such that the ring $\Set{\nu \in \C: \varepsilon \leq \absolute{\nu - \nu_0} \leq 2\varepsilon}$ contains no spatial eigenvalues. From \cite[page 68]{Kato-95}, the associated spectral projection is holomorphic, and one can construct a basis $(V_1(\lambda), \dots, V_p(\lambda))$ for its range, composed of holomorphic vectors (not necessarily eigenvectors though), e.g. by extracting suitable columns from the matrix representation of the projection.

We thus obtain two temporary family
\begin{align*}
(F_{-n}^-, \dots, F_{-1}^-, F_{1}^-, \dots, F_{n}^-),\\
(F_{-n}^+, \dots, F_{-1}^+, F_{1}^+, \dots, F_{n}^+).
\end{align*}
whose elements are solutions of \eqref{e:1st-order-pb}, holomorphic with respect to $\lambda \in \varLambda$. Furthermore the asymptotic behavior of $F_j^\pm$ at $\pm\infty$ is $\eta^\pm(F_j^\pm) = \nupm{j}$. It is direct to check that each family is in fact a basis of solutions, as  
\begin{align*}
\det(\Fp{-n}, \dots, \Fp{n})(\lambda, x) = {} & \det(\re^{A^+(\lambda)x}) \det(V_{-n}^+, \dots, V_{n}^+)(\lambda) \\
& {} + \re^{\real\left(\nup{-n} + \cdots + \nup{n}\right)(\lambda) x} \  \Cal{O}_{\lambda, x}(\re^{- n\kappa x})
\end{align*}
is non-zero for $x$ large enough (recall that $\trace{A^+(\lambda)} = \nup{-n}(\lambda) + \cdots + \nup{n}(\lambda)$).
We further refer to 
\begin{equation*}
E(\lambda) \deq \det(F_{-n}^+, \dots, F_{-1}^+, F_{1}^-, \dots, F_{n}^-)(\lambda, 0)
\end{equation*}
as an Evans function. We are now ready to complete the proof of Lemma \ref{l:spatial-eigenvalues}.

\begin{proof}[Proof of Lemma \ref{l:spatial-eigenvalues} -- part 2]
\label{proof:spatial-eigenvalues-2}
We now work with $\absolute{\lambda} \leq M$. Potentially shrinking $\theta$, Assumption \ref{a:essential-spectrum} ensure that only spatial eigenvalues associated to $\lambda_1^+ = 0$ approach the imaginary axis. A Taylor expansion of the dispersion relation together with $\alpha_1 \neq 0$ in item \textit{Single marginal point} guarantees only one spatial eigenvalue $\nup{}(0)$ has zero real part. We assume through labeling that it is either $\nup{1}$ or $\nup{-1}$. 

Assume by contradiction that it is $\nup{-1}$, meaning $\real{\nup{}(\lambda)} \leq 0$ for large values of $\lambda \in \R$. From assumption \ref{a:essential-spectrum}, $\Phi_* \deq \transpose{\left(\frac{1}{\varOmega} \uu', (\frac{1}{\varOmega} \uu')'\right)}$ is a bounded solution to \eqref{e:1st-order-pb} that can be continued for nonzero values of $\lambda$. From the assumed sign of $\nup{}(\lambda)$, $\Phi_*(\lambda) \in \Span(F_{-n}^+, \dots, F_{-1}^+) \cap \Span(F_{1}^-, \dots, F_{n}^-)$ is localized at both infinity when $\lambda \neq 0$, and thus contributes to a zero of the Evans function $E(0) = 0$. This is a contradiction with Assumption \ref{a:point-spectrum}.

Thus, for all $\nupm{j}$ except $\nup{1}$, we have $\absolute{\real{\nu(\lambda)}} \geq C$, which conclude the proof of \eqref{e:spatial-gap}. On the other hand, $\real{\nup{1}(0)} = 0$ and the contact between the spectrum and the imaginary axis is entirely described by $\nup{1}$. As mentioned above, $\nup{1}(0)$ is a simple solution of $d^+(0, \nu) = 0$ that can be holomorphically continued for small values of $\lambda$. A Taylor expansion of $\nup{1}(i\xi)$ at $\xi = \xi_0$ combined with Assumption \ref{a:essential-spectrum} assures that $\partial_\lambda \nup{1}(\xi_0)$ is real. We also claim that it is positive. Indeed, from labeling $\real{\nup{1}(\lambda)}$ becomes negative when $\lambda$ enters the interior of the spectrum $\spectrum(\cL)^\circ$ and a Taylor expansion concludes.

Fix $m>0$. By further shrinking $\theta$ if necessary, we can ensure that $\spectrum(\cL) \cap \varLambda \subset B(0, \frac{1}{2}m)$. Thus for all $\absolute{\lambda} \geq m$, $\nup{1}(\lambda)$ stays away from the imaginary axis and \eqref{e:spectral-gap-2} holds by taking a smaller $C$.
\end{proof}

\begin{lemma}
\label{l:projection}
The projection $\pi$ admits a holomorphic extension $\pi(\lambda)$ for all $\lambda \in \Lambda$. There exists $m>0$ such that the following holds. If $\absolute{\lambda} \leq m$, then $\range \pi(\lambda)$ is one dimensional, and there exists $\vp(\lambda)$, $\dvp(\lambda)$ left and right eigenvector of $\lambda - \cL[\nup{1}(\lambda)]$ with eigenvalue $0$ such that 
\begin{equation*}
\pi(\lambda) = \vp(\lambda) \dvp(\lambda).
\end{equation*}
\end{lemma}
\begin{proof}
Assumption \ref{a:point-spectrum} item \textit{Spatial Jordan block} provides a (small) $m>0$ such that when $\absolute{\lambda} \leq m$ there exist a basis of $\C^{2n \times 1}$ composed of $A^+(\lambda)$ eigenvectors. Due to the companion structure of $A^+$, these eigenvectors write as $\transpose{(v_j , \nu v_j)}$ with $\nu\in \C$ an eigenvalue, and $v_j \in \C^{n \times 1}$ in the kernel of $\lambda - \cL^+[\nu]$. Let $(\tilde{v}_k)_k$ be the dual family, defined as 
\begin{equation*}
\text{for all }\ 1 \leq \absolute{j}, \absolute{k} \leq n, 
\hspace{4em}
\scalp{v_j, \tilde{v}_k} = \ind{j = k}.
\end{equation*}
Then for all $1 \leq \absolute{j}, \absolute{k} \leq n$, 
\begin{align*}
\scalp{v_j, (\lambda - \cL^+[\nu])^* \tilde{v}_k} = {}& \scalp{(\lambda - \cL^+[\nu]) v_j, \tilde{v}_k} = 0.
\end{align*}
which implies 
\begin{equation*}
{\tilde{v}_k}^* (\lambda - \cL^{T,+}[-\nu]) = {\tilde{v}_k}^* (\lambda - \cL^+[\nu]) = 0.
\end{equation*}
As a consequence, the row vector ${\tilde{v}_1}^+ \deq {\tilde{v}_1}^*$ is a left eigenvector for $(\lambda - \cL^{T,+}[-\nu])$, which is equivalent to $({\tilde{v}_1^+}, -\nu {\tilde{v}_1^+})$ being a left eigenvector for $B(\lambda)$. We then define 
\begin{equation*}
\pi(\lambda) = v_1^+(\lambda) \cdot \tilde{v}_1^+(\lambda) \in \C^{n\times n}.
\end{equation*}
As discussed after the proof of Lemma \ref{l:ODE-behavior}, the vectors $v_1$ and $\tilde{v}_1$ can be holomorpically extended to all $\lambda \in \varLambda$. Then $\pi(\lambda)$ is the projection onto $\ker \lambda - \cL[\nup{1}(\lambda)]$. Indeed, simplicity of the eigenvalue $\nup{1}$ for $A^+(\lambda)$ ensures that this projection has one dimensional range.
\end{proof}

We now refine the above bases of solutions. In the following lemma, we restrict to $\lambda \in \varLambda$.
\begin{lemma}
\label{l:def-ODE-solution}
There exists an holomorphic basis of solutions for \eqref{e:1st-order-pb} denoted both
\begin{equation}
\label{e:bases-ODE}
(\Phip{-n}, \dots, \Phip{-1}, \Psip{1}, \dots, \Psip{n}) = (\Psim{-n}, \dots, \Psim{-1}, \Phim{1}, \dots, \Phim{n}),
\end{equation}
that satisfy $\Phim{1}(0, \cdot) = \Psip{1}(0, \cdot) = \transpose{\left(\varOmega^{-1} \uu', (\varOmega^{-1} \uu')'\right)}$, and for all $j\in \Set{1, \dots, n}$, for almost every $\lambda \in \varLambda$
\begin{equation*}
\eta^+(\Phip{-j}) = \real{\nup{-j}},
\hspace{4em}
\eta^-(\Phim{j}) = \real{\num{j}}.
\end{equation*}
Furthermore $\eta^+(\Psip{1}) = \real{\nup{1}}$, while for all $2 \leq j \leq n$ and almost every $\lambda \in \varLambda$
\begin{align*}
\eta^-(\Psim{-j}(\lambda))\ , \ \eta^-(\Psim{-1}(\lambda)) \ \leq \ -r(\lambda),
\hspace{4em}
r(\lambda) \leq \eta^+(\Psip{j}(\lambda)).
\end{align*}
\end{lemma}
\begin{proof}
From Lemma \ref{l:spatial-eigenvalues} proof, $\Phi_* \in \Span(F_{-n}^+, \dots, F_{1}^+) \cap \Span(F_{1}^-, \dots, F_{n}^-)$ at $\lambda = 0$, since $\Phi_*$ is a bounded solution. We can use either decomposition to extend $\Phi_*$ holomorphically to $\lambda \in \varLambda$, and we call both $\Psip{1}(\lambda, x)$ or $\Phim{1}(\lambda, x)$ this solution. It is then direct to check that 
\begin{align*}
(\Fp{-n}, \dots, \Fp{-1}, \Psip{1}, \Fp{2}, \dots, \Fp{n})\\
(\Fm{-n}, \dots, \Fm{-1}, \Phim{1}, \Fm{2}, \dots, \Fm{n})
\end{align*}
are still bases of solutions. Away from Jordan blocks, we further have for all $j \neq 1$
\begin{equation}
\begin{aligned}
\label{e:value-eta-ODE}
\eta^+(\Psip{1}) = \real{\nup{1}}
\qquad & \qquad
\eta^+(\Phim{1}) = \real{\num{1}}
\\
\eta^+(\Fp{j}(\lambda)) = \real{\nup{j}}, 
\qquad & \qquad
\eta^-(\Fm{j}(\lambda)) = \real{\num{j}}.
\end{aligned}
\end{equation}
From \eqref{e:value-eta-ODE} and due to the Evans function not vanishing, we conclude that 
\begin{equation}
\label{e:Evans-basis}
(\Fp{-n}, \dots, \Fp{-1}, \Phim{1}, \Fm{2}, \dots, \Fm{n}),
\end{equation}
is a basis of solutions that we label \eqref{e:bases-ODE}. To conclude the proof, it only remains to compute $\eta^+(\Psip{j}) = \eta^+(\Fm{j})$ for $j\geq 2$ and $\eta^-(\Psim{-j}) =  \eta^-(\Fp{-j})$ for $j\geq 1$. To do so, we decompose $\Fm{j}$ into the basis $(\Fp{i})_{1\leq \absolute{i} \leq n}$. Because \eqref{e:Evans-basis} is a basis, this decomposition necessarily involve at least one index $i \geq 2$, and thus leads to $\eta^+(\Fm{j})\geq r(\lambda)$ from \eqref{e:value-eta-ODE} together with Lemma \ref{l:spatial-eigenvalues}. The computation of $\eta^-(\Fp{-j})$ is similar.
\end{proof}

\subsubsection{Dual basis of solution}
We now construct solutions for \eqref{e:1st-order-trans-pb}, that have a special structure with respect to the previous basis. Let us denote
\begin{equation*}
S(x) \deq \begin{pmatrix}
-\ell_1(x) & -\ell_2\\
\ell_2 & 0 
\end{pmatrix}.
\end{equation*}
\begin{lemma}
\label{l:S-invertible}
Let $\Phi$ be a solution to \eqref{e:1st-order-pb} and $\tilde{\Phi}$ be a solution to \eqref{e:1st-order-trans-pb}. Then $\tilde{\Phi}(x) S(x) \Phi(x)$ does not depend on $x$. In addition, $S(x)$ is invertible. 
\end{lemma}
\begin{proof}
We simply differentiate this quantity, and use the ODE that $\Phi$ and $\tilde{\Phi}$ satisfy. We get
\begin{equation*}
(\dual{\Phi} S \Phi)' = \dual{\Phi}(BS + S' + SA)\Phi,
\end{equation*}
and it is direct to compute that $BS + S' + SA = 0$. Furthermore,
\begin{equation*}
S(x)^{-1} = 
\begin{pmatrix}
0 & \ell_2^{-1} \\
-\ell_2^{-1} & -\ell_2^{-1} \ell_1(x) \ell_2^{-1}
\end{pmatrix}.
\end{equation*}
\end{proof}

\begin{lemma}
\label{l:def-dual-ODE-solution}
There exists an holomorphic basis of solutions for \eqref{e:1st-order-trans-pb} denoted both
\begin{equation}
\label{e:dual-bases-ODE}
(\dPsip{-n}, \dots, \dPsip{-1}, \dPhip{1}, \dots, \dPhip{n}) = (\dPhim{-n}, \dots, \dPhim{-1}, \dPsim{1}, \dots, \dPsim{n}), 
\end{equation}
that satisfy for all $(j,k)\in \Set{1, \dots, n}^2$
\begin{equation}
\label{e:duality-relation}
\begin{aligned}
\dPsip{-j} \ S \ \Phip{-k} = \ind{j = k}, \hspace{4em}
\dPsip{-j} \ S \ \Psip{k} = 0, \\
\dPhip{j} \ S \ \Phip{-k} = 0, \hspace{4em}
\dPhip{j} \ S \ \Psip{k} = \ind{j = k}.
\end{aligned}
\end{equation}
In particular for all $j\in \Set{1, \dots, n}$
\begin{equation*}
\eta^\pm(\dPsip{-j}) = -\eta^\pm(\Phip{-j}),
\hspace{4em}
\eta^\pm(\dPhip{j}) = -\eta^\pm(\Psip{j}).
\end{equation*}
\end{lemma}
\begin{proof}
We denote $\Phip{} \in \C^{2n\times n}$ and $\Psip{} \in \C^{2n\times n}$ the matrices whose columns are the $\Phip{-j}$, and $\Psip{j}$ $1\leq j\leq n$. Similarly, we denote $\dPsip{} \in \C^{2n\times n}$ and $\dPhip{} \in \C^{2n\times n}$ the matrices whose columns are the $\dPsip{j}$, and $\dPhip{-j}$ $1\leq j\leq n$. With these notations, \eqref{e:duality-relation} rewrites
\begin{equation*}
\begin{pmatrix}
\dPsip{} \\
\dPhip{}
\end{pmatrix}
S
\begin{pmatrix}
\Phip{} & \Psip{}
\end{pmatrix}
= \id[2n\times 2n].
\end{equation*}
Since both $S$ and $\begin{pmatrix} \Phip{} & \Psip{} \end{pmatrix}$ are invertible (see Lemmata \ref{l:S-invertible} and \ref{l:def-ODE-solution}) the basis we are looking for is uniquely defined by the duality relations \eqref{e:duality-relation}. Holomorphicity also comes from \eqref{e:duality-relation}. We now compute the growth rate $\eta^+(\dPsip{-j})$, the other ones are obtained similarly.
It is direct to adapt Lemma \ref{l:ODE-behavior} to the case of the dual ODE \eqref{e:1st-order-trans-pb}, and to construct a basis of solutions with asymptotic behaviors. By decomposing $\dPsip{-j}$ into this basis, we see that there exists $\dual{V}_{-j}^+ \in \C^{1\times 2n}$ such that when $y\to +\infty$,
\begin{equation*}
\dPsip{-j}(\lambda, y) = \dual{V}_{-j}^+(\lambda) \re^{B^+(\lambda) y} + \Cal{O}_{\lambda, y}(\re^{(\dual{\eta} - \frac{\kappa}{2}) \absolute{y}}), 
\end{equation*}
where for convenience we wrote $\dual{\eta} = \eta^+(\dual{V}_{-j}^+ \re^{B^+ \cdot}) = \eta^+(\dPsip{-j})$. We further denote $S^+ = \lim_{x\to +\infty} S(x)$, and recall the basis $(V_k^+(\lambda))_{1\leq \absolute{k}\leq n}$ from the proof of Lemma \ref{l:def-ODE-solution}. We claim that there exists an index $k$ such that
\begin{equation}
\label{e:computation-rate}
\eta^+(\dual{V}_{-j}^+ \re^{B^+ x} \, S^+ \, \re^{A^+ x} V_k^+) = \eta^+(\dual{V}_{-j}^+ \re^{B^+ x}) + \eta^+(\re^{A^+ x} V_k^+) \deq \dual{\eta} + \eta.
\end{equation}
Indeed, by denoting $E_1 \subset \C^{1 \times 2n}$ the eigenspace of $B^+$ associated to $\eta^+(\dual{V}_{-j}^+ \re^{B^+ x})$, and because $S^+$ has full range, there exists an eigenspace $E_2 \subset \C^{2n \times 1}$ for $A^+$ such that $E_1 S^+ E_2 \neq \Set{0}$. Because the $(V_k^+)_k$ respects the spectral projections, it is then possible to choose $k$ such that $\eta^+(\re^{A^+ x} V_k^+)$ is associated to $E_2$. This show the first equality of \eqref{e:computation-rate}.

By choice of $S$, the function $x\mapsto \dual{V}_{-j}^+ \re^{B^+ x} \, S^+ \, \re^{A^+ x} V_k^+$ is constant and we further denote it $C$. From \eqref{e:computation-rate}, we deduce that $C\neq 0$, leading to $\dual{\eta} + \eta = 0$. Then, introducing the two asymptotic behavior into \eqref{e:duality-relation}, we get when $x\to +\infty$
\begin{align*}
\ind{-j = k} = {} & \dPsip{-j} S \Phip{k}, \\
= {} & \dual{V}_{-j}^+ \re^{B^+ x} \, S^+ \, \re^{A^+ x} V_k^+ + \Cal{O}(\re^{(\tilde{\eta} + \eta - \frac{\kappa}{2}) x}),\\
= {} & C + \Cal{O}(\re^{-\frac{\kappa}{2} x}).
\end{align*}
This last equality ensures that $-j = k$, thus providing $\eta = \eta^+(\Phip{-j})$.
\end{proof}

Remark that we labeled the dual basis so that the $\dPhipm{j}$ are exponentially decaying at $\pm \infty$, while the $\dPsipm{j}$ are exponentially growing at $\pm \infty$.

\subsection{Resolvent kernels}
\label{ss:construction-resolvent-kernel}
We now study the resolvent kernel, that is the solution $x\mapsto K(\lambda, x, y) \in \C^{n\times n}$ to 
\begin{equation}
\label{e:resolvent-kernel-pb}
(\lambda - \cL(x)) K = \delta_y \, \id[n\times n],
\end{equation}
where $y \in \R$ is fixed, $\delta_y$ stands for the Dirac distribution at point $x = y$, and $\id[n\times n]$ is the $n\times n$ identity matrix. In this section, we construct $K$, and then decompose it in preparation of the decomposition of $\re^{t\cL}$. 

\subsubsection{Construction} To construct $K$, we will use the ODE solutions $\Phipm{j}$, $\Psipm{j}$, $\dPhipm{j}$ and $\dPsipm{j}$ defined in the previous section, see \eqref{e:bases-ODE} and \eqref{e:dual-bases-ODE}. Remark that the expression of $K$ will only depends on $\phipm{j}$, $\psipm{j}$, $\dphipm{j}$ and $\dpsipm{j}$, which are defined as the first $n$ components of the previous $2n$-size solutions. However, we need to work with the $2n$-size solutions for the jump condition at $x = y$ to contain enough information.

\begin{lemma}
\label{l:adjoint-kernel}
The function $y\mapsto K(\lambda, x, y)$ satisfies the transpose resolvent problem \eqref{e:trans-res-pb}.
\end{lemma}
\begin{proof}
Let $H(\lambda, x, y)$ be the resolvent kernel for the adjoint operator $\cL^*$, that is the solution to 
\begin{equation*}
(\lambda - \cL^*(x)) H = \delta_y.
\end{equation*}
By definition of the adjoint operator, for any $v_1, v_2 \in L^2(\R, \C^n)$,
\begin{equation*}
\scalp{(\lambda - \cL)v_1, v_2} = \scalp{v_1, (\lambda - \cL)^* v_2},
\end{equation*}
where $\scalp{.,.}$ denotes the duality inner product inherited from $L^2(\R, \C^n)$, with $\scalp{v_1, v_2} \in \C$. 
Denoting $e_j = \transpose{(0, \dots, 1, \dots, 0)}$ the $j$-th element of the canonical basis of $\C^n$, we choose $v_1 = K(\lambda, \cdot, y) e_k$ and $v_2 = H(\bar{\lambda}, \cdot, x) e_j$ in the above equality, to obtain\footnote{Although the chosen $v_1$, $v_2$ are not necessarily in $L^2$, this substitution is licit since it holds on every compact set $[-R, R]$.}
\begin{equation*}
\scalp{\delta_y e_k, H(\bar{\lambda}, \cdot, x) e_j} = \scalp{K(\lambda, \cdot, y)e_k, \delta_x e_j}.
\end{equation*}
This equality reads $H(\bar{\lambda}, y, x) = K(\lambda, x, y)^*$ where $K^* = \transpose{\overline{K}}$ stands for the adjoint matrix. Thus, 
\begin{equation*}
(\bar{\lambda} - \cL^*(y)) K(\lambda, x, \cdot)^* = \delta_x,
\end{equation*}
where exceptionally, derivatives in $\cL^*(y)$ refers to the second spatial variable of $K$. Taking the conjugate transpose of this matrix valued equation, we get that $K$ satisfies \eqref{e:trans-res-pb}.
\end{proof}

\begin{lemma}
Let $\K = \begin{pmatrix}
K & \partial_y K\\
\partial_x K & \partial_{xy} K
\end{pmatrix}$. Seen as a function of $x$, the jump of $\K$ at $y$ is
\begin{equation}
\label{e:jump-cond}
\jump{\K}_y = \begin{pmatrix}
0 & \ell_2^{-1} \\
-\ell_2^{-1} & -\ell_2^{-1} \ell_1(y) \ell_2^{-1}
\end{pmatrix} = S(y)^{-1}.
\end{equation}
\end{lemma}
\begin{proof}
The proof is exactly the one of \cite[Lemma 4.6]{Howard-Zumbrun-98}.
\end{proof}

\begin{proposition}
There exists a unique solution to \eqref{e:resolvent-kernel-pb} that is holomorphic in $\lambda \in \varLambda$ (punctually for each $(x, y) \in \R^2$) and lying in $L^\infty(\R_x, L^1(\R_y))$ (punctually for each $\lambda \in \varLambda \backslash \spectrum(\cL)$). It writes as
\begin{equation}
\label{e:expression-K}
K(\lambda, x, y) = 
\begin{cases}
\displaystyle \sum_{j = 1}^n \phim{j}(\lambda, x) \, \dphip{j}(\lambda, y), \hspace{4ex} & x<y,\\[1.5em]
\displaystyle \sum_{j = 1}^n \phip{-j}(\lambda, x) \, \dphim{-j}(\lambda, y), & y<x.
\end{cases}
\end{equation}
\end{proposition}
\begin{proof}
We first prove that such a $K$ necessarily satisfies \eqref{e:expression-K}. When $x \neq y$, each column of $\K$ satisfies \eqref{e:1st-order-pb}, and thus decompose into the basis \eqref{e:bases-ODE}. Because $x\mapsto \K(\lambda, x, y)$ is bounded, this decomposition only involves the $\Phip{-j}(x)$ when $x \in (y, +\infty)$, and only the $\Phim{j}(x)$ when $x \in (-\infty, y)$. 
Thus, there exists $N_j(\lambda, y) \in \C^{2n\times 1}$ such that 
\begin{equation}
\label{e:expression-K-1}
\K(\lambda, x, y) = 
\begin{cases}
\sum_{j = 1}^n \Phim{j}(\lambda, x) \, N_{j}(\lambda, y),\hspace{4ex} & x<y,\\[1em]
\sum_{j = 1}^n \Phip{-j}(\lambda, x) \, N_{-j}(\lambda, y), & y<x.
\end{cases}
\end{equation}
From Lemma \ref{l:adjoint-kernel}, we see that when $y\neq x$, each row of $\K$ satisfies \eqref{e:1st-order-trans-pb}. Multiplying \eqref{e:expression-K-1} on the left by $\dPsim{j_0}(x)S(x)$, for a fixed $x$, we see that $N_{j_0}(\lambda,y)$ also satisfies \eqref{e:1st-order-trans-pb} for $y> x$, and thus decompose into the basis \eqref{e:dual-bases-ODE}. Because $y\mapsto \dPsim{j_0}(x)S(x) \K(\lambda,x,y)$ is bounded, this decomposition only involves the $\dPhip{j}(y)$. With similar argument for the $y<x$, case, we get that there exists coefficients $m_{j,k}(\lambda) \in \C$ such that
\begin{equation}
\label{e:expression-K-2}
\K(\lambda, x, y) = 
\begin{cases}
\sum_{j = 1}^n \Phim{j}(\lambda, x) \, m_{j,k}(\lambda) \, \dPhip{k}(\lambda, y),\hspace{4ex} & x<y,\\[1em]
\sum_{j = 1}^n \Phip{-j}(\lambda, x) \, m_{-j,-k}(\lambda) \, \dPhim{-k}(\lambda, y), & y<x.
\end{cases}
\end{equation}
To obtain \eqref{e:expression-K}, it remains to show that $m_{j,k}(\lambda) = \ind{j = k}$, and then to extract the top left $n\times n$ matrix from the $2n\times 2n$ equation \eqref{e:expression-K-2}.

To do so, we rewrite \eqref{e:expression-K-2} in a more compact form. Let $M_+ = (m_{j,k})_{1\leq j, k \leq n} \in \C^{n\times n}$, and $M_-$ defined similarly from the $m_{-j,-k}$. We further denote $\Phim{} \in \C^{2n\times n}$ the matrix whose columns are the $\Phim{j}$, $1\leq j\leq n$, and $\dPhip{} \in \C^{n\times 2n}$ the matrix whose rows are the $\dPhip{k}$, $1\leq k\leq n$. Then from  
\begin{equation*}
\Phim{j} \, m_{j,k} \, \dPhip{k} = (\Phim{} \, e_j) \, m_{j,k} \,  (\transpose{e_k} \, \dPhip{}) = \Phim{} (e_j m_{j,k} \transpose{e_k}) \dPhip{}, 
\end{equation*}
we see that
\begin{equation*}
\K(\lambda, x, y) = 
\begin{cases}
\Phim{}(\lambda, x) \, M_+(\lambda) \, \dPhip{}(\lambda, y),\hspace{4ex} & x<y,\\[1em]
\Phip{}(\lambda, x) \, M_-(\lambda) \, \dPhim{}(\lambda, y), & y<x.
\end{cases}
\end{equation*}
Thus, the jump condition \eqref{e:jump-cond} rewrites 
\begin{equation*}
\begin{pmatrix}
\Phip{} &\Phim{}
\end{pmatrix}
\begin{pmatrix}
M_- & 0 \\
0 & -M_+
\end{pmatrix}
\begin{pmatrix}
\dPhim{} \\ \dPhip{}
\end{pmatrix} = S^{-1}, 
\end{equation*}
or equivalently 
\begin{equation*}
\begin{pmatrix}
M_- & 0 \\
0 & -M_+
\end{pmatrix}
= 
\left(
\begin{pmatrix}
\dPhim{} \\ \dPhip{}
\end{pmatrix}
S
\begin{pmatrix}
\Phip{} &\Phim{}
\end{pmatrix}
\right)^{-1} = \id[2n\times 2n],
\end{equation*}
by construction of the dual basis. Thus \eqref{e:expression-K} is proved, and its right hand side is a holomorphic bounded solution to \eqref{e:resolvent-kernel-pb} by the previous construction.
\end{proof}

\subsubsection{Decomposition}
We now decompose $K$ into several holomorphic parts, each one selected according to its behavior at $\lambda = 0$. Recall that $\phim{1}(0, x) = \varOmega(x)^{-1} \uu'(x)$ from Lemma \ref{l:def-ODE-solution}.
\begin{definition}
For $\lambda \in \varLambda$, and $(x,y) \in \R^2$ we let
\begin{align*}
\ki(\lambda, x, y) & {} = \re^{-\nup{1}(\lambda) (y - x)} \, \ind{y > x},
\end{align*}
and
\begin{align*}
\Ki(\lambda, x, y) & {} = \pi \, \ki(\lambda, x, y) \\[1em]
\Kii(\lambda, x, y) & {} = \big(\pip(\lambda) - \pi\big) \ki(\lambda, x, y), \\[1em]
\Kiii(\lambda, x, y) & {} = \psip{1}(\lambda, x) \dphip{1}(\lambda, y) \, \ind{y > x} - \pip(\lambda) \ki(\lambda, x, y), \\[1em]
\Kiv(\lambda, x, y)  & {} = K(\lambda, x, y) - \bigg(\Ki(\lambda, x, y) + \Kii(\lambda, x, y) + \Kiii(\lambda, x,y)\bigg).
\end{align*}
\end{definition}

\subsection{Bounds}
\label{ss:bound-resolvent-kernel}
\begin{proposition}[Close to the origin]
\label{p:small-lambda}
Recall the expression \eqref{e:def-rate} of $r$ and let $m>0$ from Lemma \ref{l:spatial-eigenvalues}. There exists $C>0$ such that for almost every $\lambda \in B(0,m)$, the following holds.
For all $(x,y)\in \R^2$,
\begin{align*}
\absolute{\Kii(\lambda, x, y)} & {} \leq C \absolute{\lambda}\re^{-\real{\nup{1}(\lambda)} (y - x)} \ind{y > x}, \\
\absolute{\Kiii(\lambda, x, y)} & {} \leq C \re^{-\real{\nup{1}(\lambda)} (y - x)} \bigg(\min\left(1, \re^{-\frac{\kappa}{2} x}\right) + \min\left(1, \re^{-\frac{\kappa}{2} y}\right) \bigg) \ind{y > x}, \\
\absolute{\Kiv(\lambda, x, y)} & {} \leq C \re^{-r(\lambda) \absolute{x-y}}.
\end{align*}
Furthermore, for all $x < y < 0$
\begin{equation*}
\absolute{\Ki(\lambda, x, y)} + \absolute{\Kiii(\lambda, x, y)} + \absolute{\Kiii(\lambda, x, y)} \leq C \re^{-r(\lambda) \absolute{x-y}}.
\end{equation*}
\end{proposition}
\begin{proof}
The bound on $\Kii$ is direct using smoothness of $\lambda \mapsto \pip(\lambda)$. From the definition of $\pip$, $\nup{1}$, $\vp$ and $\dvp$,
\begin{align}
\label{e:bound-Kii} \Kiii(\lambda, x, y) = {}& \bigg(\psip{1}(\lambda, x) \dphip{1}(\lambda, y) - \vp(\lambda) \dvp(\lambda) \re^{-\nup{1}(\lambda)(y - x)}\bigg) \ind{y>x}
\end{align}
is a sum of $\cO$ terms in \eqref{e:ODE-behavior}, hence contributing to $x$ or $y$ exponential decay at $+\infty$ as claimed.
We now turn to the estimate on $\Kiv$, and recall from Lemma \ref{l:def-ODE-solution} that $\psip{1} = \phim{1}$. Thus
\begin{align}
\label{e:expr-Kiv}
\Kiv(x, y) =
\begin{cases}
\sum_{j = 2}^n \phim{j}(x) \, \dphip{j}(y), \hspace{4ex} & x<y, \\[1em]
\sum_{j = 1}^n \phip{-j}(x) \, \dphim{-j}(y), & y<x.
\end{cases}
\end{align}
We demonstrate how to bound terms that present a $\ind{x<y}$, the procedure adapts for the other ones. Remark that the rest of the proof is valid for all $\lambda \in \varLambda$. To bound each $\phi$ and $\dual{\phi}$, we rely on Lemmas \ref{l:def-ODE-solution} and \ref{l:def-dual-ODE-solution} together with Lemma \ref{l:spatial-eigenvalues}. We avoid the finite amount of $\lambda$ giving rise to Jordan block in the matrix $A$, see Assumption \ref{a:point-spectrum}, and obtain an estimate on the $\phi$ and $\dual{phi}$ from their exponential rate: there is no deterioration by a polynomial factor.
\begin{itemize}
\item $x < 0 < y$: for all $j \in \Set{2, \dots, n}$, we use $\real{\num{j}(\lambda)} \geq r(\lambda)$
\begin{equation*}
\absolute{\phim{j}(x) \, \dphip{j}(y)} \leq C \, \re^{\real{\num{j}(\lambda)} x} \, \re^{-r(\lambda) y} \leq C \, \re^{r(\lambda)(x-y)}.
\end{equation*}
\item $0 < x < y$: for all $j \in \Set{2, \dots, n}$, we use that $\eta^+(\phim{j}) = \eta^+(\psip{j}) = - \eta^+(\dphip{j})$ to get
\begin{equation*}
\absolute{\phim{j}(x) \, \dphip{j}(y)} \leq C \, \re^{\eta^+(\phim{j}) (x-y)} \leq C \, \re^{r(\lambda)(x-y)}.
\end{equation*}
\item $x < y < 0$: for all $j \in \Set{2, \dots, n}$, we have $\eta^-(\dphip{j}) = - \eta^-(\psip{j}) = - \eta^-(\phim{j}) = -\real{\num{j}}$. Thus
\begin{equation}
\label{e:bound-K-II}
\absolute{\phim{j}(x) \, \dphip{j}(y)} \leq C \, \re^{\real{\num{j}(\lambda)} (x - y)} \leq C \, \re^{r(\lambda)(x-y)}.
\end{equation}
\end{itemize}
We conclude with the improved bound on $\Ki$, $\Kii$ and $\Kiii$ when both $x$ and $y$ are negative. Remark that \eqref{e:bound-K-II} also holds true with $j = 1$. It provides the claimed bound.
\end{proof} 

\begin{proposition}[Away from the origin]
\label{p:large-lambda}
Let $m>0$ from Lemma \ref{l:spatial-eigenvalues}. For almost every $\lambda \in \varLambda \backslash  B(0,m)$, and for all $(x,y)\in \R^2$, 
\begin{equation*}
\absolute{K(\lambda, x, y)} \leq C\re^{-r(\lambda) \absolute{x-y}}.
\end{equation*}
\end{proposition}
\begin{proof}
We restrict to points $\lambda$ with no spatial Jordan block, so that the bounds on $\phi$ and $\dual{\phi}$ still hold.
The bound on $\Ki$, $\Kii$ and $\Kiii$ follow the same argument as in Proposition \ref{p:small-lambda}, together with $\real{\nup{1}(\lambda)} \geq r(\lambda)$ from Lemma \ref{l:spatial-eigenvalues}.
The estimates on $\Kiv$ is already made in the proof of Proposition \ref{p:small-lambda}.
\end{proof}

\section{Linear estimates}

\subsection{Green kernel bounds in weighted space}
We obtain the Green kernel $G$ associated to the parabolic equation $\partial_t u = \cL u$ through inverse Laplace transform. Let $\varGamma = \kappa + \ri\R$ for some large enough $\kappa > 0$, then 
\begin{equation*}
G(t, x, y) = \frac{1}{2i\pi} \int_\varGamma \re^{\lambda t} K(\lambda, x, y) \d \lambda.
\end{equation*}
We define accordingly $\gi$ to $\Giv$ from the resolvent kernels $\ki$ to $\Kiv$, so that
\begin{equation*}
G(t, x, y) = \gi(t,x,y) \pi + \Gii(t,x,y) + \Giii(t,x,y) + \Giv(t,x,y).
\end{equation*}
From the above, all resolvent kernels are analytic in the simply connected region $\varLambda$, so that the path $\varGamma$ can be homotopically deformed in this region. Extra care should be taken when moving its  \guillemet{endpoints}. We show in appendix \ref{s:appendix} that inverse Laplace transform remains valid in the current setting when $\varGamma$ is replaced by a sectorial path. To ensure that estimates from the previous section apply, we also want $\varGamma$ to avoid points with spatial Jordan blocks. This is always possible at no cost because they are isolated and away from $0$.
\newline 

We now define $\varGamma$ as a simple sectorial path. Other choices are possible, see \cite[p. 816]{Howard-Zumbrun-98}. For non negative constants $\alpha$ and $\varXi$ let
\begin{equation*}
\begin{aligned}
\varGamma_0 \deq {} & \Set{\gamma(\xi) \deq  \alpha_1 \ri (\xi - \xi_0) - \alpha (\xi - \xi_0)^2: \xi \in [\xi_0 - \varXi, \xi_ + \varXi]}, \\[1em]
\varGamma_\infty \deq {} & \Set{\gamma(\xi) \deq \alpha_1 \ri (\xi - \xi_0) - \alpha {\varXi}^2 - \frac{\theta}{2} \absolute{\xi - \xi_0 - \varXi} : \xi \in \R \backslash (\xi_0 - \varXi, \xi_0 + \varXi)}.
\end{aligned}
\end{equation*}
and $\varGamma \deq \varGamma_0 \cup \varGamma_\infty$.

\begin{lemma}
\label{l:Taylor-nup1}
Fix $\alpha \in \left[0, \frac{\real{\alpha_2}}{2}\right)$. Then there exists $\varXi>0$ such that $\varGamma \cap \spectrum(\cL) = \Set{0}$ and $\varGamma_0 \subset B(0, m)$, with $m$ as in Proposition \ref{p:small-lambda}. Furthermore, there exists $\beta_0 \in \C$ with $\real{\beta_0} = 0$, $\sigma \in (0, +\infty)$ and $\beta_2 \in \C$ with $\real{\beta_2}>0$ such that
\begin{equation*}
\nup{1}(\gamma(\xi)) = \beta_0 + \frac{\alpha_1}{\sigma}\ri (\xi - \xi_0) + \beta_2 (\xi - \xi_0)^2 + \Bigo_{\xi \to \xi_0}(\big(\xi - \xi_0\big)^3).
\end{equation*}
\end{lemma}
\begin{proof}
By taking $\varXi$ sufficiently small (depending on $\theta$), Assumptions \ref{a:essential-spectrum} and \ref{a:point-spectrum} ensures that $\varGamma\backslash \Set{0}$ does not intersect $\spectrum(\cL)$. Shrinking $\varXi$ if necessary, we get $\varGamma_0 \subset B(0, m)$. Taylor expanding $\nup{1}\circ \gamma$ at $\xi_0$ we get for any $\alpha \geq 0$:
\begin{equation*}
\nup{1}(\gamma(\xi)) - \nup{1}(0) = \ri \alpha_1 (\xi - \xi_0) \, {\nup{1}}'(0) - (\xi - \xi_0)^2 \left(\alpha \, (\nup{1})'(0) + {\alpha_1}^2 \frac{{\nup{1}}''(0)}{2}\right) + \bigo\big((\xi-\xi_0)^3\big).
\end{equation*}
We set $\beta_0 = \nup{1}(0)$ and $\frac{1}{\sigma} = (\nup{1})'(0)$, the latter being real and positive from Lemma \ref{l:spatial-eigenvalues}.
When $\alpha < \frac{\real{\alpha_2}}{2}$ and $\varXi$ is small enough, $\gamma(\xi)$ is outside the spectrum, thus $\real{\nup{1}(\gamma(\xi))} \geq 0$ and 
\begin{equation}
\label{e:Tailor-nup1-bound}
\real{\beta_2} = - \alpha \, (\nup{1})'(0) - {\alpha_1}^2 \frac{\real{(\nup{1})''(0)}}{2} > 0.
\end{equation}
\end{proof}

\begin{proposition}
\label{p:green-kernel-bounds-1}
There exists $\kappa>0$ and $C>0$ such that, denoting 
\begin{align}
\nonumber 
g(t, x, y) = {}& \frac{1}{\sqrt{t + y - x}}\exp\left(-\kappa \frac{(y - x - \sigma t)^2}{t + y - x}\right), 
\end{align}
the following estimates hold for all $(x, y) \in \R^2$ and $t \geq 0$:
\begin{align*}
\absolute{\gi(t,x,y)} \leq {}& C \, g(t, x, y) \, \ind{y > x}, \\[1em]
\absolute{\Gii(t,x,y)} \leq {}& C \min\left(1, \frac{1}{\sqrt{t + y - x}}\right) g(t, x, y) \, \ind{y > x}, \\[1em]
\absolute{\Giii(t,x,y)} \leq {}& C \, g(t, x, y) \, \bigg( \exp^{0, -\kappa}(x) + \exp^{0, -\kappa}(y)\bigg) \, \ind{y > x}, \\[1em]
\absolute{\Giv(t,x,y)} \leq {}& C \, \re^{-\kappa t} \re^{-\kappa \absolute{x-y}}.
\end{align*}
Furthermore for all $x < y < 0$ and $t \geq 0$, 
\begin{equation*}
\absolute{\Gi(t,x,y)} + \absolute{\Gii(t,x,y)}  + \absolute{\Giii(t,x,y)} \leq C \, \re^{-\kappa t} \re^{-\kappa \absolute{x-y}}.
\end{equation*}
\end{proposition}
\begin{proof}
High frequencies. Using Proposition \ref{p:large-lambda}, high frequencies in $\absolute{G(t,x,y)}$ are damped quickly enough to be irrelevant:
\begin{align*}
\Absolute{\int_{\varGamma_\infty} \re^{\lambda t} K(\lambda, x, y) \d \lambda} \leq {} & \re^{-\alpha_1 \varXi^2 t} \int_{\absolute{\xi} \geq \varXi} \re^{- \frac{\theta}{2} \absolute{\xi - \varXi} t} \, \absolute{K(\gamma(\xi), x, y)} \, \absolute{\gamma'(\xi)} \, \d \xi,\\
\leq {}& C \re^{-\alpha_1 \varXi^2 t} \int_{\varXi}^{+\infty} \re^{-\frac{\theta}{2} \absolute{\xi - \varXi} t} \, \re^{-r(\gamma(\xi))\absolute{x - y}} \, \d \xi,\\
\leq {}& C \re^{-\alpha_1 \varXi^2 t} \, \re^{- r(\gamma(\varXi)) \absolute{x - y}}.
\end{align*}

Bound on $\gi$. The proof is mainly about extracting the relevant behavior from the integrand. The expansion of $\nup{1}$ in Lemma \ref{l:Taylor-nup1} ensures there exists a holomorphic $h_1:B(0, \kappa) \to \C$ satisfying $\absolute{h_1(\lambda)}\leq C\absolute{\lambda}^3$, and such that
\begin{equation*}
\ki(\gamma(\xi), x, y) = \re^{-(\beta_0 + \frac{\alpha_1}{\sigma} \ri (\xi - \xi_0) + \beta_2 (\xi - \xi_0)^2) (y - x)} \, \re^{h_1(\ri(\xi - \xi_0)) y - x} \, \ind{y>x}.
\end{equation*}
We fix a small constant $\kappa > 0$, and assume that $\varXi \leq \kappa$. We denote 
\begin{align*}
H_1(\ri \zeta, y - x) = {}& \re^{h_1(\ri \zeta) (y - x)}, \\[0.5em]
H_2(\ri \zeta) = {}& \gamma'(\zeta) = \ri (\alpha_1 + 2 \alpha \ri \zeta), \\[0em]
H_3(\ri \zeta, y - x, t) = {}& \exp\bigg(- \beta_0 (y - x) - (1 - \kappa) \frac{\alpha_1}{\sigma} \ri \zeta (y - x - \sigma t)\bigg) 
\end{align*}
and compute 
\begin{align*}
\int_{\varGamma_0} \re^{\lambda t} \ki(\lambda, x, y) \d \lambda = {}& \int_{-\varXi}^{\varXi} \re^{(\ri \alpha_1 \zeta - \alpha \zeta^2) t} \re^{-(\beta_0 + \frac{\alpha_1}{\sigma} \ri \zeta + \beta_2 \zeta^2) (y - x)} \big(H_1\times H_2\big) (\ri \zeta, y - x) \, \ind{y > x} \, \d \zeta, \\
= {}& \int_{-\varXi}^{\varXi} \re^{- \kappa \frac{\alpha_1}{\sigma} \ri \zeta \left(y - x - \sigma t\right)} \re^{-\zeta^2\left(\beta_2 (y - x) + \alpha_1 t\right)} \, \big(H_1\times H_2 \times H_3\big) (\ri \zeta, y - x, t) \, \ind{y > x} \, \d \zeta.
\end{align*}
We further write
\begin{align*}
z_1 = {}& y - x - \sigma t , \\
z_2 = {}& \beta_2 (y - x) + \alpha t,
\end{align*}
and the above computation provides $H:B(0, \kappa) \times (0,+\infty)^2 \to \C$, holomorphic with respect to its first argument, satisfying
\begin{equation}
\label{e:k1-claim}
\int_{\varGamma_0} \re^{\lambda t} \ki(\lambda, x, y) \d \lambda = \int_{-\varXi}^{\varXi} \re^{- \kappa \frac{\alpha_1}{\sigma} \ri \zeta z_1} \re^{- \zeta^2 z_2} H(\ri \zeta, y - x, t) \ind{y > x} \d \zeta.
\end{equation}
Completing the square
\begin{equation*}
- \kappa \frac{\alpha_1}{\sigma} \ri\zeta z_1 - \zeta^2 z_2 = z_2 \left(\ri \zeta - \kappa \frac{\alpha_1 z_1}{2 \sigma z_2}\right)^2 - \frac{\kappa^2 {\alpha_1}^2}{4 \sigma^2} \frac{{z_1}^2}{z_2}, 
\end{equation*}
we see that 
\begin{equation}
\label{e:k1-bound1}
\Absolute{\int_{\varGamma_0} \re^{\lambda t} \ki(\lambda) \d \lambda} = \exp\left(-\frac{\kappa^2 {\alpha_1}^2}{4 \sigma^2} \frac{{z_1}^2}{\real{z_2}}\right) \ \Absolute{\int_{-\varXi}^{\varXi} \exp\left(z_2 \left(\ri \zeta - \kappa \frac{\alpha_1 z_1}{2 \sigma z_2}\right)^2\right) H(\ri \zeta) \ind{y>x} \d \zeta}.
\end{equation}
The right hand side integral can be estimated by shifting the complex integration path
\begin{align}
\label{e:k1-bound2}
\Absolute{\int_{-\varXi}^{\varXi} \exp\left(z_2 \left(\ri \zeta - \kappa \frac{\alpha_1 z_1}{2 \sigma z_2}\right)^2\right) H(\ri \zeta) \d \zeta} = {} & \Absolute{\int_{-\varXi}^{\varXi} \re^{- z_2 \zeta^2} H\left(\ri \zeta + \kappa\frac{\alpha_1 z_1}{2 \sigma z_2}\right) \d \zeta}.
\end{align}
Because $\absolute{\alpha_1 \frac{z_1}{2 \sigma z_2}}$ is bounded uniformly with respect to $t\geq 0$ and $\absolute{x - y} \geq 0$, it is enough to bound the function $H(\cdot, x, y) \ind{y>x}$ on a $\kappa$-neighborhood of the origin. For this we remark that
\begin{align*}
(y - x) \Absolute{h_1\bigg(\ri \zeta + \kappa \frac{\alpha_1 z_1}{2 \sigma z_2}\bigg)} 
\leq {}& C (y - x) \Absolute{\ri \zeta + \kappa \frac{\alpha_1 z_1}{2 \sigma z_2}}^3, \\
\leq {}& C (y - x)  \bigg(\absolute{\zeta}^3 + \Absolute{\kappa \frac{\alpha_1 z_1}{2 \sigma z_2}}^3\bigg), \\
\leq {}& C \kappa \, \real{z_2} \, \zeta^2 + C \kappa^3 \frac{{z_1}^2}{\real{z_2}},
\end{align*}
where the last inequality comes from $\absolute{\zeta} \leq \kappa$, $(y - x) \leq C\real{z_2}$ and $\absolute{z_1} \leq C\absolute{z_2}$.
Since $H_2$ does not depend on $t$, $x$ and $y$, it is obviously bounded on $B(0, C\kappa)$. A simple computation leads to
\begin{align*}
\Absolute{H_3\left(\ri \zeta + \kappa \frac{\alpha_1 z_1}{2 \sigma z_2}, y - x, t\right)} = {}& \exp\left(- \kappa(1 - \kappa) \frac{\alpha_1}{2 \sigma^2} \, \frac{{z_1}^2}{\real{z_2}}\right)\leq 1.
\end{align*}
All in all, taking $\kappa$ small enough, we have
\begin{equation*}
\Absolute{H\left(\ri \zeta + \kappa \frac{\alpha_1 -z_1}{2 \sigma z_2}, y - x, t\right)} \leq C\exp\left(\frac{\real{z_2}}{2} \zeta^2\right) \exp\left(\kappa^2\frac{{z_1}^2}{8\real{z_2}}\right).
\end{equation*}
Inserting this bound into \eqref{e:k1-bound2}, and recalling the Gaussian factor from \eqref{e:k1-bound1}, we arrive at
\begin{align}
\nonumber\Absolute{\int_{\varGamma_0} \re^{\lambda t} \ki(\lambda, y) \d \lambda} \leq {}& C \exp\left(-\kappa^2 \frac{{z_1}^2}{8 \real{z_2}}\right) \int_{-\varXi}^{\varXi} \re^{-\frac12 \real{z_2} \zeta^2} \d \zeta, \\
\label{e:bound-Gi-diffusive-decay}\leq {}& C \frac{1}{\sqrt{\real{z_2}}} \exp\left(-\kappa^2 \frac{{z_1}^2}{8 \real{z_2}}\right).
\end{align}
Relabelling $\kappa$ to obtain a $\frac{\kappa}{t + \absolute{x - y}}$ instead of the current $\frac{\kappa^2}{8 \real{z_2}}$, we are done with this step.

Bound on $\Gii$ and $\Giii$ are obtained by slight modifications of the previous approach. For $\Gii$, the extra $\absolute{\gamma(\xi)} \leq \absolute{\zeta}$ factor  improves estimate \eqref{e:bound-Gi-diffusive-decay} by an additional $\sqrt{\real{z_2}}$ factor:
\begin{equation*}
\int_{-\varXi}^{\varXi} \absolute{\zeta} \re^{-\frac12 \real{z_2} \zeta^2} \d \zeta \leq C \min\left(1, \frac{1}{\real{z_2}}\right).
\end{equation*}
For $\Giii$, the extra spatial localization go through estimates without changes.

Bond on $\Giv$ and improved bounds when $x < y < 0$. For these terms, we homotopically deform $\varGamma_0$ into 
\begin{equation*}
\varGamma_\mathrm{s} = \Set{\gamma_\mathrm{s}(\xi) = \ri \alpha_1 (\xi-\xi_0) - \alpha {\varXi}^2  : \xi \in [\xi_0 - \varXi, \xi_0 + \varXi]}
\end{equation*}
which belongs to the left half of the complex plane. This allows to bound 
\begin{equation*}
\Absolute{\int_{\varGamma_\mathrm{s}} \re^{\lambda t} \Giv(\lambda, x,y) \d \lambda} \leq \int_{-\varXi}^{\varXi} \re^{-\alpha_1 {\varXi}^2 t} \absolute{\Kiv(\gamma_\mathrm{s}(\xi), x, y)} \absolute{\gamma_\mathrm{s}'(\xi)} \d \xi.
\end{equation*}
Inserting the estimates on $\Kiv$ from Proposition \ref{p:small-lambda}, we then get the claimed estimate.
\end{proof}

With small modification, it is possible to bound temporal and spatial derivatives of the Green kernel. Since there is a smoothness defect at $x = y$, space-differentiation creates Dirac masses. To properly state the bounds, we decompose the kernel derivatives into the sum of two distribution. The first is atomic, the second is absolutely continuous with respect to Lebesgue measure of $\R_y$. For $(n, k, j) \in \Set{0,1}^3$, we thus write 
\begin{equation*}
\partial_t^n \partial_x^k \partial_y^j \gi(t,x,y) = \gi_{\at, n,k,j}(t,x) \, \Dirac_{x}(y) + \gi_{\ac, n, k, j}(t,x,y).
\end{equation*}
Only $x$-derivatives will be useful for the main result proof, we estimate the other for completeness.
\begin{proposition}
\label{p:green-kernel-bounds-2}
There exists $\kappa>0$ such that for all $(n, k, j) \in \Set{0,1}^2$, for all $(x, y) \in \R^2$ and $t \geq 0$,  
\begin{align*}
\absolute{\gi_{\ac, n, k, j}(t,x,y)} \leq {}& C (t + \absolute{x - y})^{- \frac{n + k + j}{2}} g(t, x, y) \, \ind{y > x},\\[1em]
\absolute{\gi_{\at, n, k, j}(t,x)} \leq {}& C \re^{-\kappa t},\\[1em]
\absolute{\Giv_{\ac, n, k, j}(t,x,y)} \leq {}& C \re^{-\kappa t} \re^{-\kappa \absolute{x-y}},\\[1em]
\absolute{\Giv_{\at, n, k, j}(t,x)} \leq {}& C \re^{-\kappa t}.
\end{align*}
The estimates on $\gi$ are also satisfied by $\Gii$ and $\Giii$.
\end{proposition}
\begin{proof}
Bound on $\partial_t^n \partial_y^l \gi$. The proof is almost identical to the previous one, so we simply point at differences. 
Differentiating with respect to time brings an extra $\lambda$ factor below the integral, that we directly transfer into the function $H$. The improved decay then comes from 
\begin{equation*}
\int_\R \absolute{\zeta}^n \re^{-z_2 \zeta^2} \d \zeta \leq C {z_2}^{-\frac{1}{2}(1+n)}.
\end{equation*}

When differentiating with respect to $y$, we pass the $\partial_y$ below the integral, and compute that
\begin{equation*}
\partial_y \big(\dphip{1}(\lambda, y) \ind{y>x} \big) = \partial_y \dphip{1}(\lambda, y) \, \ind{[x, \infty)}(y) + \dphip{1}(\lambda, y) \, \Dirac_{x}(y).
\end{equation*}
Lemma \ref{l:ODE-behavior} provides a description of $\partial_y \dphip{1}(\lambda)$, through the $n$ last components of $\dPhip{1}$. 
More precisely, because $\dual{V}_1^+(\lambda)$ is a left eigenvector for $B^+(\lambda)$ and due to the special structure of $B$, see \eqref{e:1st-order-trans-pb}, we get that $\dual{V}_1^+(\lambda) = \begin{pmatrix} v_1^+(\lambda) & \nup{1}(\lambda) v_1^+(\lambda) \end{pmatrix}$. Thus, $\partial_y \dphip{1}(\lambda,y)$ writes as 
\begin{equation}
\label{e:decomp-partial-y-dphip1}
\partial_y \dphip{1}(\lambda, y) = \nup{1}(\lambda) \re^{- \nup{1}(\lambda) y} v_1^+(\lambda) + \bigo_{\lambda, y}\bigg(\re^{-\left(\real{\nup{1}(\lambda)}+\kappa\right) y}\bigg).
\end{equation}
In the right hand side of \eqref{e:decomp-partial-y-dphip1}, the bracket term is $y$-integrable, uniformly with respect to $\lambda$ close to $0$. This allows to deform the integration path $\varGamma_0$ into the left half of the complex plane for this term. Remark that $\psip{1}(\lambda, x)$ also needs to be estimated on this deformed integration path. Fortunately, it becomes exponentially localized when entering the spectrum.
\begin{align*}
\Absolute{\int_{\varGamma_0} \re^{\lambda t} \, \psip{1}(\lambda, x) \, \bigo_{\lambda, y}\left(\re^{-\left(\real{\nup{1}(\lambda)}+\kappa\right) y}\right) \, \ind{y>x} \, \d \lambda} \leq {}& C \re^{-\eta t} \, \re^{- \eta (x - y)} \, \re^{-\kappa y} \, \ind{y>x}, \\
\leq {}& C \re^{-\eta t} \, \re^{-\frac{\eta}{2} (y - x)} \, \ind{y>x}.
\end{align*} 
The term with the Dirac mass is controlled similarly: Since there is no need to care for integrability at $y\to +\infty$, we can stabilize $\varGamma_0$ at no cost. It only remains to bound 
\begin{equation*}
\int_{\varGamma_0} \re^{\lambda t} \psip{1}(\lambda, x)\, \nup1(\lambda) \re^{-\nup1(\lambda) y} v_1^+(\lambda) \d \lambda 
\end{equation*}
and for this part we follow the strategy that was used to bound $\gi$. The extra $\nup1(\gamma(\xi)) = \bigo(\xi-\xi_0)$ factor is saved into the function $H_1$, and used to improve decay as described previously.

Differentiation with respect to $x$ is done similarly. Estimates on $\Gii$ and $\Giii$ are proven in the same way.

For the estimates on $\partial_t^n \partial_y^j \Giv$, the integration path can be moved in the stable complex half-plane without difficulty.
\end{proof}

\begin{proposition}
\label{p:lebesgue-bound-G}
There exists positive constants $\kappa$, $\sigma$ and $C$ such that the following points hold.
\begin{itemize}
\item For all $\rho: \R \to [1,+\infty)$ non-decreasing, for all $x \in \R$ and $t\geq 1$
\begin{equation}
\label{e:estimate-gi-rho}
\Norm{\frac{\gi(t,x, \cdot)}{\rho}}_{W^{k,1}(\R)} \leq \frac{C}{\rho\left(x + \frac\sigma2 t\right) t^{\frac{k}{2}}},
\end{equation}
while
\begin{align*}
\Norm{\gi(t,x, \cdot)}_{W^{k, \infty}(\R)} + \Norm{\Gii(t,x, \cdot)}_{W^{k, 1}(\R)} \leq {}& \frac{C}{t^\frac{1 + k}{2}}, \\[1em]
\Norm{\Giii(t,x, \cdot)}_{W^{k, 1}(\R)} \leq {}& \frac{C \exp^{0, -\kappa}(x)}{t^\frac{k}{2}} \\[1em]
\Norm{\Giv(t,x, \cdot)}_{W^{k,1} \cap W^{k,\infty}(\R)} \leq {}& C \re^{-\kappa t}.
\end{align*}
\item For all $x \leq 0$ and $t\geq 1$
\begin{align*}
\Norm{\gi(t,x, \cdot)}_{W^{k,1} \cap W^{k,\infty}(-\infty, 0)} + \Norm{\Gii(t,x, \cdot)}_{W^{k,1} \cap W^{k,\infty}(-\infty, 0)} + \Norm{\Giii(t,x, \cdot)}_{W^{k,1} \cap W^{k,\infty}(-\infty, 0)} \leq C \re^{-\kappa t}.
\end{align*}
\item For all $t\in [0,1]$ and all $x\in \R$
\begin{equation*}
\norm{G(t, x, \cdot)}_{L^1(\R)} \leq C.
\end{equation*}
\end{itemize}
\end{proposition}
\begin{proof}
We start with the $L^1(\R)$ estimate of $\gi$, and decompose the half line into three parts: 
\begin{equation*}
[x, +\infty) = \Set{x} + \left[0, \frac{\sigma}{2} t\right] \cup \left[\frac{\sigma}{2}t , 2\sigma t\right] \cup \left[2\sigma t, +\infty\right).
\end{equation*} 
On the second part, we exploit the fact that the Gaussian-like kernel is translated, while $\rho$ is not. On the first and third part, we simply use the Gaussian tails to get strong temporal decay. Starting with the second part, we perform the change of variable $z = y - x - \sigma t$, and then bound
\begin{align*}
\int_{-\frac{\sigma}{2} t}^{\sigma t} \frac{1}{\sqrt{t + z + \sigma t}} \, \frac{1}{\rho(z + x + \sigma t)} & \exp\left(-\kappa \frac{z^2}{z + (1 + \sigma)t}\right) \d z, \\
\leq {} & \frac{1}{\sqrt{t}}\, \frac{1}{\rho(x + \frac{\sigma}{2} t)} \int_{-\frac{\sigma}{2} t}^{\sigma t}  \exp\left(-\frac{\kappa}{1 + 3\sigma} \, \frac{z^2}{t}\right) \d z ,\\
\leq {} & \frac{C}{\rho(x + \frac{\sigma}{2} t)}.
\end{align*}
We now turn to the Gaussian tails. For the left one, the same change of variable leads to 
\begin{align*}
\int_{-\sigma t}^{-\frac{\sigma}{2} t} \frac{1}{\sqrt{t + z + \sigma t}} \, \frac{1}{\rho(z + x + \sigma t)} \exp\left(-\kappa \frac{z^2}{z + (1 + \sigma)t}\right) \d z \leq \int_{-\sigma t}^{-\frac{\sigma}{2} t}  \exp\left(-\frac{\kappa}{1 + 2\sigma} \, \frac{z^2}{t}\right) \d z.
\end{align*}
The right hand side then decays exponentially in time.
For the right tail, the same approach gives
\begin{align*}
\int_{\sigma t}^{+\infty} \frac{1}{\sqrt{t + z + \sigma t}} \, \frac{1}{\rho(z + x + \sigma t)} & \exp\left(-\kappa \frac{z^2}{z + (1 + \sigma)t}\right) \d z \\
\leq {}& \int_{\sigma t}^{+\infty}  \exp\left(-\kappa \frac{z^2}{z + (1 + \sigma)t}\right) \d z, \\
\leq {}& \int_{\sigma t}^{+\infty}  \exp\left(-\kappa q z \right) \d z.
\end{align*}
To obtain the last bound, we set $q \in (0, \frac{\sigma}{1 + 2\sigma})$ and check that $z^2 \geq q z (z + (1 + \sigma)t)$ holds. 

We now prove the other estimates. The $L^\infty$ estimate on $\gi$ is direct. The $L^1$ estimates on $\Gii$ and $\Giii$ is proven similarly as the $\gi$ estimate, using either the extra $\sqrt{\real{z_2}}$ factor to obtain temporal decay or the additional spatial localization $\ind{y>x} \exp^{0, -\kappa}(y) \leq \exp^{0, -\kappa}(x)$. The estimate on $\Giv$ as well as the improved exponential decay are direct to obtain from the improved bounds of Proposition \ref{p:green-kernel-bounds-1}. The short time estimate follows from $\norm{g(t, x, \cdot)}_{L^1(0, +\infty)} \leq C$.
\end{proof}

\subsection{Semigroup bounds}
The semigroup is obtained through
\begin{equation*}
(\re^{t\cL} u_0)(x) = \int_\R G(t,x, y) u_0(y) \d y,
\end{equation*}
and we define in a similar way operators $\si(t)$ to $\Siv(t)$ from $\gi$ to $\Giv$, so that
\begin{equation}
\label{e:semigroup-decomp}
\re^{t\cL} = \si(t)\pi + \Sii(t) + \Siii(t) + \Siv(t).
\end{equation}
Because $\uu'$ is the eigenfunction of $\cA$ associated to $\lambda = 0$, it is direct to compute that $\re^{t\cA} \uu' = \uu'$ is not decaying in time. The aim of the above decomposition is to isolate similar non-decaying modes for an initial data $u_0$ that is not necessarily equal to $\uu'$.

We start with an improved bound when the initial data is sufficiently localized.
\begin{proposition}
\label{p:improved-decay}
Let $\kappa > 0$ be sufficiently small and $k\in \Set{0, 1}$. There exists $\eta>0$ such that for $(j,n) \in \Set{0,1}^2$, $1 \leq p \leq \infty$ and $t \geq 0$,
\begin{equation*}
\Norm{\exp^{0, \kappa} \, \re^{t \cL} u_0}_{W^{k, p}(\R)} \leq C\re^{-\eta t} \, \norm{\exp^{0, \kappa} u_0}_{W^{k, p}(\R)}.
\end{equation*}
\end{proposition}
\begin{proof}
This estimate is based on the fact that using additional weights $\exp^{0, -\kappa}$ further stabilizes the spectrum of $\cL$, thus leading to exponential temporal decay from spectral gap. In the following, we briefly describe how the material from previous sections can be used to obtain such statement. 

First introduce the weighted operator $\cT \deq \exp^{0,\kappa} \cL \exp^{0,-\kappa}$, and denote $\mu_j^{\pm}(\lambda)$ its spatial eigenvalues $j\in \Set{-n, \dots, -1, 1, \dots, n}$, that are the solutions to the dispersion relations
\begin{equation*}
\det\left(\lambda - \cT^\pm[\mu] \right) = 0.
\end{equation*}
Then notice by standard computation that $\cT^+[\mu] = \cL^+[\mu+\kappa]$ and $\cT^-[\mu] = \cL^-[\mu]$, respectively ensuring that $\mu_j^+ = \nup{j} + \kappa$ and $\mu_j^- = \num{j}$. This last expression allow to obtain a bound as \eqref{e:spatial-gap} on all $\mu_j^\pm$, including $\mu_1^+$, provided that $\kappa$ is small enough. Then, all material from section \ref{ss:ODE-solution} up to here straightforwardly adapt except that $\nombre{1}$, $\nombre{2}$ and $\nombre{3}$ terms now behave as $\nombre{4}$-terms. In particular, one can use integral paths in $\Set{\lambda \in \C : \real(\lambda)\leq -\eta}$ for $\eta$ small enough, thus leading to the claimed exponential temporal decay when $k = 0$. 

When $k = 1$, the estimates holds true for short time $t \in [0,1]$ from existence result of parabolic semigroups \cite{Pazy-83}. When $t\geq 1$, we use the improved Proposition \ref{p:green-kernel-bounds-2}. 
\end{proof}

\begin{proposition}
\label{p:temp-decay}
There exists positive constants $\sigma$ and $C$ such that for all $p \in [1, +\infty]$, for all $\rho:\R \to [1,+\infty)$ non-decreasing, for all $(n,k,j) \in \Set{0,1}^3$ and all $t \geq 0$,
\begin{equation*}
\Norm{\partial_t^n \, \partial_x^k \, \re^{t\cL} \, \partial_y^j v_0}_{L^p(\R)} \leq \frac{C}{t^\frac{n + k + j}{2}} \bigg( \frac{1}{\rho(\sigma t / 4)} \Norm{\rho \, \pi \, v_0}_{L^p(\R)} + \frac{1}{\sqrt{1 + t}} \Norm{v_0}_{L^p(\R)} \bigg).
\end{equation*}
\end{proposition}
\begin{proof}
We start with $n = k = j = 0$, which corresponds to an $L^1$-control of the Green kernel. The key point to handle $\Si$ and $\Siii$ is to use the semigroup property $\re^{t \cL} \, v_0 = \re^{\frac{t}{2} \cL} \, \re^{\frac{t}{2} \cL} \, v_0$. We first look at the semigroup from $0$ to $t/2$. From Proposition \ref{p:lebesgue-bound-G}, we easily obtain estimates for each term in \eqref{e:semigroup-decomp}. To control $\Si$ and $\Siii$, we must treat differently $(-\infty, 0)$ and $(0,+\infty)$ to get rid of the spatial dependence (for exemple in \eqref{e:estimate-gi-rho}):
\begin{align}
\nonumber
\Norm{\Si(t/2) v_0}_{L^\infty((0,+\infty))} \leq {}& \sup_{y\geq 0} \Norm{\frac{\gi(t/2, y, \cdot)}{\rho}}_{L^1(\R)} \norm{\rho \, \pi \, v_0}_{L^\infty(\R)}, \\
\label{e:intermediate-bound-Si}
\leq {}& \frac{C}{\rho(\sigma t / 4)} \norm{\rho \, \pi \, v_0}_{L^\infty(\R)}, \\[1em]
\nonumber
\Norm{\Si(t/2) v_0}_{L^\infty((-\infty, 0))} \leq {}& \sup_{y\leq 0} \Norm{\Gi(t/2, y, \cdot)}_{L^1(\R)} \norm{v_0}_{L^\infty(\R)} \leq \norm{v_0}_{L^\infty(\R)}.
\end{align}
Similarly 
\begin{equation*}
\Norm{\re^{\kappa \cdot} \, \Siii(t/2) v_0}_{L^\infty((0,+\infty))} + \Norm{\Siii(t/2) v_0}_{L^\infty((-\infty, 0))} \leq  C \norm{v_0}_{L^\infty(\R)}.
\end{equation*}
Estimates on $\Sii$ and $\Siv$ are straightforward:
\begin{align*}
\Norm{\Sii(t/2) v_0}_{L^\infty(\R)} \leq {}& \frac{C}{\sqrt{1 + t/2}} \norm{v_0}_{L^\infty(\R)}, \\[1em]
\Norm{\Siv(t/2) v_0}_{L^\infty(\R)} \leq {}& C \re^{-\kappa t} \norm{v_0}_{L^\infty(\R)}.
\end{align*}
With these estimates at hand, we are now ready to control the semigroup from $t/2$ to $t$. For example \eqref{e:intermediate-bound-Si} ensures that
\begin{equation*}
\Norm{\re^{\frac{t}{2} \cL} (\ind{y\geq 0} \, \Si(t/2) \, v_0)}_{L^\infty(\R)} \leq \frac{C}{\rho(\sigma t / 4)} \norm{\rho \, \pi \, u_0}_{L^\infty(\R)},
\end{equation*}
while the improved decay of Proposition \ref{p:lebesgue-bound-G} together with $x < y$ provides
\begin{equation*}
\Norm{\re^{\frac{t}{2} \cL} (\ind{y\leq 0} \, \Si(t/2) \, v_0)}_{L^\infty(\R)} \leq C \re^{-\kappa t} \norm{v_0}_{L^\infty(\R)}.
\end{equation*}
To control $\Siii$ we proceed similarly and use Proposition \ref{p:improved-decay}:
\begin{align*}
\Norm{\re^{\frac{t}{2}\cL} \Siii(t/2) v_0}_{L^\infty(\R)} \leq {}& C \re^{-\eta \frac{t}{2}} \bigg(\Norm{\re^{\kappa \cdot} \, \Siii(t/2) v_0}_{L^\infty((0,+\infty))} + \Norm{\Siii(t/2) v_0}_{L^\infty((-\infty, 0))}\bigg), \\[1em]
\leq {}& C \re^{-\eta \frac{t}{2}} \norm{v_0}_{L^\infty(\R)}.
\end{align*}
The estimates of $\Sii$ and $\Siv$ straightforwardly follow from $\sup_{x\in \R} \norm{G(t,x,\cdot)}_{L^1(\R)} < +\infty$:
\begin{align*}
\Norm{\re^{\frac{t}{2} \cL} \, \Sii(t/2) \, v_0}_{L^\infty(\R)} \leq {}& \frac{C}{\sqrt{1 + t/2}} \norm{v_0}_{L^\infty(\R)}, \\
\Norm{\re^{\frac{t}{2} \cL} \, \Siv(t/2) \, v_0}_{L^\infty(\R)} \leq {}& C \re^{-\kappa \frac{t}{2}} \norm{v_0}_{L^\infty(\R)}.
\end{align*}
Gathering all four estimates, this case is complete.

When $n = 1$, the time derivative hits either of the two kernels $G(t/2, x, y)$ or $G(t/2, y, z)$, improving temporal decay according to Proposition \ref{p:green-kernel-bounds-2}. When $j = 1$, we integrate by parts with respect to space:
\begin{equation*}
\int_\R G(t/2, y, z) \partial_z u_0 \d z = - \int_\R \partial_z G(t/2, y, z) u_0 \d z,
\end{equation*}
again improving temporal decay. When $k = 1$ and $t\geq 1$, the $x$-derivative hits the first Green kernel improving decay as stated in Proposition \ref{p:lebesgue-bound-G}.
\end{proof}

\section{Nonlinear study}
We want to construct classical solutions to \eqref{e:main} in $\buc{0}$ spaces. For $T> 0$, let 
\begin{equation*}
X_T \deq \cC^1((0,T], \buc{1}(\R, \C^n) \ \cap \ \cC^0([0,T], \buc{3}(\R, \C^n).
\end{equation*}
We then decompose the solution to \eqref{e:main} as
\begin{equation}
\label{e:correction-ansatz}
u(t, x + ct + x_\infty) = \uu(x) + \varOmega(x) v(t, x).
\end{equation}
We will construct the phase deformation $x_\infty$ and the shape deformation $v$ from the initial value $u(0, \cdot)$.

\begin{lemma}
Recall notation \eqref{e:Taylor-expansion} and let $T > 0$. Assume that $v \in X_T$ satisfies
\begin{equation}
\label{e:correction-dynamic}
v_t = \cL v + \cN(\cdot, v),
\end{equation}
with 
\begin{equation*}
\cN(x, v) = \frac{1}{\varOmega(x)} \bigg( \big(T_{2, f}(\uu(x), \varOmega(x) v)\big)_x + T_{2, g}(\uu(x), \varOmega(x) v) \bigg), 
\end{equation*}
Then \eqref{e:correction-ansatz} provides $u \in X_T$ satisfying \eqref{e:main}. 
\end{lemma}
\begin{proof}
Inserting \eqref{e:correction-ansatz} into \eqref{e:main}, we directly obtain \eqref{e:correction-dynamic}. This computation is licit since $v\in X_T$.
\end{proof}

\begin{lemma}
\label{l:non-linear-bound}
There exists $\kappa>0$ such that 
\begin{equation*}
\Norm{\frac{1}{\varOmega} \cN(\cdot, v)}_{L^\infty(\R)} \leq C\norm{v}_{W^{1,\infty}(\R)}^2. 
\end{equation*}
\end{lemma}
\begin{proof}
Since $\varOmega \cN$ is the sum of second order Taylor expansions at $\uu$ in the direction $\varOmega v$, we get $C>0$ such that uniformly on $\R$
\begin{equation*}
\absolute{\varOmega \cN} \leq C \left(\absolute{\varOmega v}^2 + \absolute{\varOmega v_x}^2 \right), 
\end{equation*}
leading to the claimed bound.
\end{proof}
\begin{lemma}
\label{l:local-existence}
Let $v_0 \in \buc{2}(\R)$. There exists $T>0$ and a unique $v\in X_T$ that solves \eqref{e:correction-dynamic} with initial data $v(0, \cdot) = v_0$. Furthermore, it satisfies
\begin{equation}
\label{e:Duhamel}
v(t) = \re^{t \cL} v_0 + \int_0^t \re^{(t - s) \cL} \cN(\cdot, v(s)) \d s.
\end{equation}
\end{lemma}
\begin{proof}
As a differential operator whose coefficients lie in $\buc{3}(\R)$, the unbounded operator $\cL$ is closed from $\buc{3}(\R)$ to $\buc{1}(\R)$ with dense domain. In addition, there exists $\lambda_0 \in (0,+\infty)$ such that the resolvent bound
\begin{equation}
\label{e:resolvent-bound-cL}
\Norm{(\lambda - \cL)^{-1}}_{\buc{1}(\R) \to \buc{1}(\R)} \leq \frac{1}{\lambda - \lambda_0}
\end{equation}
holds on a sector. Indeed, when $v \in \buc{1}(\R)$ such that $\norm{v}_{W^{1,\infty}(\R)} = 1$ and $\norm{v}_{L^1(\R)} = 1$, the numerical range estimate
\begin{align*}
\real{\scalp{\cL v, v}}_{L^2(\R)} = {}& -\scalp{d v_y, v_y} + \norm{\ell_0 v^2}_{L^1(\R)}, \\
\leq {}& \norm{\ell_0}_{L^\infty(\R)}.
\end{align*}
holds, and implies the resolvent bound \eqref{e:resolvent-bound-cL}, see \cite[Chap. I, Theorem 3.9]{Pazy-83}. Those two properties guarantee that $\cL$ is the infinitesimal generator of an analytic semigroup \cite[Proposition 2.1.1]{Lunardi-13}.

On the other hand, the map $v \mapsto \cN(\cdot, v)$ is uniformly Lipschitz continuous from $\buc{1}(\R)$ to $\buc{0}(\R)$. 
From this point, the parabolic theory concludes the proof: See \cite[Theorems 7.1.2]{Lunardi-13} with $\alpha = \frac{1}{2}$.
\end{proof}

\begin{lemma}
\label{l:blow-up}
Let $T_\mathrm{max}$ be the maximal time satisfying Lemma \ref{l:local-existence}. If $T_\mathrm{max} < +\infty$, then 
\begin{equation*}
\limsup_{t \to T_\mathrm{max}} \norm{v(t)}_{L^\infty} = +\infty.
\end{equation*}
\end{lemma}
\begin{proof}
If $T_\mathrm{max}$ is finite, \cite[Proposition 7.1.8]{Lunardi-13} implies that $\norm{\cN(\cdot, v(t))}_{L^\infty(\R)}$ diverges when $t \to T_\mathrm{max}$. Lemma \ref{l:non-linear-bound} concludes the proof.
\end{proof}

\begin{lemma}
\label{l:estimate-shape-function}
There exists positive constants $\kappa$, $\sigma$, $a$ such that the following holds. Let $v_0 \in \buc{2}(\R)$, and $\rho$ a sub-exponential weight with constants $(\kappa/\sigma, M)$ satisfying $\rho(\sigma t/4) \leq \sqrt{1 + t}$. There exists positive constants $C_1$ and $C_2$ such that: If
\begin{equation*}
\tilde{E}_{0,\rho} \deq \norm{\rho \pi v_0}_{W^{1, \infty}(\R)} + \norm{v_0}_{W^{1, \infty}(\R)}
\end{equation*}
is finite, then
\begin{equation*}
\varTheta_\rho(t) \deq \sup_{s \in [0,t]} \rho(\sigma s)\norm{v(s)}_{W^{1, \infty}(\R)}
\end{equation*}
is continuous and satisfies for all $t\geq 0$
\begin{align}
\label{e:bootstrap-1}
\varTheta_\rho(t) \leq {}& C_1 \left( \tilde{E}_{0,\rho} + M {\varTheta_\rho(t)}^2\right), \\[0.5em]
\label{e:bootstrap-2}
\varTheta_\rho(t) \leq {}& C_2 \big( \tilde{E}_{0,\rho} + M \varTheta_\rho(t) \varTheta_1(t)\big).
\end{align}
\end{lemma}
\begin{proof}
When $t \in [0, 2]$, we can choose $C_i$ and $M$ so large that the estimates hold true. When $t\geq 2$, we use \eqref{e:Duhamel} to estimate $\varTheta$. The linear term is handled with Proposition \ref{p:temp-decay}. The integral term is handled using Proposition \ref{p:improved-decay} and Lemma \ref{l:non-linear-bound}.  
\begin{align*}
\norm{\partial_x^k v(t)}_{L^\infty(\R)} \leq {}& \frac{C}{\rho(\sigma t)} \tilde{E}_{0,\rho} + C \int_0^t \frac{\re^{-\kappa (t - s)}}{(t - s)^\frac{k}{2}} \norm{\exp^{0, \delta} \cN(\cdot, v(s))}_{L^\infty(\R)} \d s, \\
\leq {}& \frac{C}{\rho(\sigma t)} \tilde{E}_{0,\rho} + C\int_0^t \frac{\re^{-\kappa (t - s)}}{(t - s)^\frac{k}{2}} \frac{1}{\rho(\sigma s)} \varTheta_\rho(s) \varTheta_1(s) \d s, 
\end{align*}
The integral term is decomposed as $\int_0^t = \int_0^{t - 1} + \int_{t - 1}^t$, and estimated using Definition \ref{d:sub-exp}:
\begin{align*}
\int_0^{t - 1} \frac{\re^{-\kappa (t - s)}}{(t - s)^\frac{k}{2}} \frac{1}{\rho(\sigma s)} \d s \leq \int_0^{t - 1} \re^{-\kappa (t - s)} \frac{1}{\rho(\sigma s)} \d s \leq \frac{C}{\rho(\sigma (t - 1))}, \\
\int_{t - 1}^t \frac{\re^{-\kappa (t - s)}}{(t - s)^\frac{k}{2}} \frac{1}{\rho(\sigma s)} \d s \leq \frac{1}{\rho(\sigma (t - 1))} \int_{t - 1}^t (t - s)^{-\frac{k}{2}} \d s \leq \frac{C}{\rho(\sigma (t - 1))}. 
\end{align*}
Remark that $s\mapsto \rho(\sigma s)$ is a sub-exponential weight with constant $\kappa$. We now show that the map $\varrho : t \mapsto \frac{\rho(\sigma t)}{\rho(\sigma (t - 1))}$ is bounded. We assume by contradiction it is unbounded on $[1, +\infty)$ and construct a sequence $(t_n)_{n \geq 0}$ as follows. Let $t_0 = 2$. For all $n\geq 1$ we set 
\begin{equation*}
C_n \deq \max(n, \sup_{[1, t_{n - 1} + 1]} \varrho).
\end{equation*}
By the running assumption, there exists $t_n$ such that $\varrho(t_n) \geq C_n$. As a consequence, we have for all $n\geq 1$
\begin{align*}
& \rho(\sigma(t_n)) \geq n \ \rho(\sigma(t_n - 1)), \\
& t_n - 1 \geq t_{n - 1}.
\end{align*}
Because $\rho$ is increasing, it implies $\rho(\sigma t_n) \geq n! \  \rho(\sigma t_0)$, which is a contradiction with Definition \ref{d:sub-exp}. Due to $\varrho$ being bounded, we easily deduce
\begin{equation*}
\rho(\sigma t) \norm{\partial_x^k v(t)}_{L^\infty(\R)}\leq C \left( \tilde{E}_{0,\rho} + M \varTheta_\rho(t)\varTheta_1(t)\right),
\end{equation*}
which conludes the proof.
\end{proof}

\begin{lemma}
\label{l:bootstrap}
Let $a$, $b$ be positive constants such that $a < \frac{1}{4b}$. If $t\mapsto h(t)$ is a continuous and positive map such that $h(0) \leq 2a$ and
\begin{equation*}
h(t) \leq a + b h(t)^2,
\end{equation*}
then for all $t\geq 0$, $h(t) \leq 2a$.
\end{lemma}
\begin{proof}
Studying the sign of $h \mapsto a + bh^2 - h$ on $[0,2a]$, and using continuity of the later map, it is direct to see that $h(t)$ can not leave $[0,2a]$ from its right boundary.
\end{proof}

\begin{proof}[Proof of Theorem \ref{t:main}.]
We first prove item \ref{i:th-1}. Existence of the solution locally in time is proven in Lemma \ref{l:local-existence}, and we now prove that $T$ is infinite. Lemma \ref{l:estimate-shape-function} provides $C_1 > 0$ such that for all $t\in [0,T]$ estimate \eqref{e:bootstrap-1} holds with $\rho = 1$. Setting $\delta = \frac{1}{4{C_1}^2}$ allows to apply Lemma \ref{l:bootstrap}, ensuring that $\varTheta_1(t) \leq 2 C_1 E_0$ for all $t\in [0,+\infty)$. This provides the estimate claimed in \ref{i:th-1}, as well as $T = +\infty$ from Lemma \ref{l:blow-up}.

We turn to \ref{i:th-2}. As a first step, we will construct $x_\infty$. From Lemma \ref{l:front}, we easily compute
\begin{equation*}
\lim_{x\to +\infty} \frac{\uu(x - x_\infty) - U_+}{\varOmega(x)} = \re^{(\kappa - \ri \xi_0) x_\infty} \, \vp(0).
\end{equation*}
Using the notation $b \vp(0) = \lim_{+\infty} \pi \frac{u_0 - \uu}{\varOmega}$, we deduce that there exists a unique $x_\infty \in [-\delta, \delta]$ such that 
\begin{equation*}
\pi \lim_{x\to +\infty} \frac{u_0(x) - \uu(x - x_\infty)}{\varOmega} =  (b + 1 - \re^{\kappa x_\infty}) \vp(0) = 0.
\end{equation*}
Indeed, $b > -1$ follows from the smallness assumption of $E_0$.

This choice of $x_\infty$ ensures that $v_0$ defined from \eqref{e:correction-ansatz} satisfies $\lim_{x \to +\infty} \pi  v_0(x) = 0$. 

We then proves that there exists a sub exponential weight $\rho$ with constants $(\delta, M)$ such that $\norm{\rho \pi v_0}_{L^\infty(\R)} < +\infty$. Following \cite{Garenaux-Rodrigues-25}, we define a recursive sequence by $x_0 = 0$ and for all $j \geq 0$
\begin{equation*}
x_{j + 1} \deq \inf\Set{x \geq x_j + 1 \, : \, \forall y > x, \hspace{0.5em} \pi v_0(y) \leq \frac{E_0}{2^{j+1}}},
\end{equation*}
and then let 
\begin{equation*}
\rho \deq \ind{(-\infty, 0]} + \sum_{j\geq 0} 2^{j + 1} \, \ind{(x_j, x_{j+1}]}.
\end{equation*}
These definitions ensure that $E_{0, \rho} = \norm{\rho \pi v_0}_{L^\infty} \leq 2 E_0$, and we now show that $\rho$ is sub-exponential with claimed constants. Remember from Lemma \ref{l:exemple-sub-exp-weights} that we only need to check \eqref{e:sub-exp-alternative-estimate}. We let $J: \R \to \N_{\geq 0}$ be the map defined by $J((x_j, x_{j + 1}]) = \Set{j+1}$. For $0 \leq y_1 \leq y_2$ we remark that 
\begin{equation*}
y_2 - y_1 \geq x_{J(y_2) - 1} - x_{J(y_1)} \geq J(y_2) - 1 - J(y_1),
\end{equation*}
the second estimate is easy to prove by induction. As a consequence, 
\begin{equation*}
\frac{\rho(y_2)}{\rho(y_1)} = \frac{2^{J(y_2)}}{2^{J(y_1)}} \leq 2 \re^{(y_2 - y_1)\ln(2)}, 
\end{equation*}
and $\rho$ is a sub-exponential weight with constants $(\ln(2), 2)$. Remark that the $\ln(2)$ constant can be replaced by $\ln(p)$ for any $p\in (1, +\infty)$ by a straightforward adaptation of the above. This step is completed.

From the above discussion, we see that a proof of \ref{i:th-3} will imply \ref{i:th-2}, and we conclude with a proof of the former. We repeat the same approach as for \ref{i:th-1}. We may shrink $\delta$ to $\kappa/\sigma$ so that Lemma \ref{l:bootstrap} applies. We further shrink $\delta$ to the value $\frac{1}{4 C_1 C_2 M}$, and \eqref{e:bootstrap-2} together with $\varTheta_1(t) \leq 2C_1 E_0$ imply
\begin{equation*}
\varTheta_\rho(t) \leq \frac{C_2 E_{0, \rho}}{1 - 2 C_1 C_2 M E_0}.
\end{equation*}
This is precisely \ref{i:th-3}, and the proof is complete.
\end{proof}

\section{Fourier and Evans modes}
We assumed above that only one parameter modulation was necessary to prove stability. For many interesting systems however, the element $0 \in \varSigma(\cL)$ has higher multiplicity. Some of its eigenmodes are associated to parameter invariance of \eqref{e:main}. Typical known examples are translation and rotation invariance. Other eigenmodes may arise from wave parameter variation: the background wave $\uu$ depend on additional parameters that appear in \eqref{e:main} itself. Known examples are wave number, wave speed and temporal frequency. 

Since our approach does not directly apply for wave parameter variation, we may assume in the next paragraphs that the $0$ eigenvalue is only associated to system invariants. When these correspond to essential spectrum curves, the above decomposition of the resolvent kernel should readily adapt, and lead to a stability result up to parameter modulation.

System invariants may also be linked to zeros of the Evans function at $0$:
\begin{equation}
\label{e:Evans-mode}
\phi(0, \cdot) \in \Span_j (\phi^-_j(0,\cdot)) \, \cap \, \Span_j (\phi^+_{-j}(0,\cdot)). 
\end{equation}
When this is the case, we distinguish between two scenario:
\begin{itemize}
\item If the Evans modes are localized
\begin{equation*}
\phi \text{ satisfies }\eqref{e:Evans-mode} 
\hspace{1.5em}
\implies
\hspace{1.5em}
\exists \kappa>0, 
\hspace{0.5em} \absolute{\phi(0, x)} \leq C\re^{-\kappa \absolute{x}}
\end{equation*}
then they do not change localization in a neighborhood of $0$, and it is possible to define spectral projections there. In that case, one can easily separate between Evans modes \eqref{e:Evans-mode} and Fourier modes 
\begin{equation}
\label{e:Fourier-mode}
\phi(0, x) \sim \re^{\ri \xi_0 x} \hspace{3em} x\to +\infty.
\end{equation}
For an Evans mode, the parameter modulation can be handled using the standard spectral projection approach \cite{Howard-Zumbrun-98}. For Fourier modes, the parameter dynamics is carried following Theorem \ref{t:main} proof.
\item Since $0$ lies at the essential spectrum border, one may argue that Evans modes may not be localized, but only bounded
\begin{equation*}
\phi \text{ satisfies }\eqref{e:Evans-mode} 
\hspace{1.5em}
\implies
\hspace{1.5em}
\absolute{\phi(0, x)} \leq C.
\end{equation*}
And discrimination between Evans and Fourier modes is not direct anymore.
\end{itemize}

\begin{theorem}
\label{t:Fourier-Evans}
Assume that $\varSigma_\textrm{ess}(\cL)$ does not intersect $\Set{\real{\lambda}>0}$. Assume that the marginal group velocities are directed towards the wave interface: With the same notation as in \eqref{e:marginal-lambda-expansion}, assume
\begin{equation}
\label{e:incoming-group-velocities}
\lambda^\pm_k(\ri \xi_0) = 0 
\hspace{2em}
\implies
\hspace{2em}
\pm \imag{ \partial_\nu \lambda^\pm_k (\ri \xi_0)} > 0. 
\end{equation}
Then Fourier and Evans modes are disjoint: If there exists a nonzero bounded solution to \eqref{e:pvp}, then it can not satisfy \eqref{e:Evans-mode} and \eqref{e:Fourier-mode} simultaneously.
\end{theorem}
\begin{proof}
We discuss the case of a marginal spectrum caused by $\cL^+$, the discussion adapts easily if it comes from $\cL^-$. Assume by contradiction that a bounded $\phi$ solves \eqref{e:pvp}, and satisfy both \eqref{e:Evans-mode} - \eqref{e:Fourier-mode}.
As a consequence, there exists $j_0$ such that  $\nu^+_{j_0}(0) = \ri \xi_0$. The definition of the dispersion relation provides a $k_0$ such that $\lambda^+_{k_0}(\ri \xi_0) = 0$, and thus $\imag{ \partial_\nu \lambda^+_{k_0} (\ri \xi_0)} > 0$. If they are several such $\phi$ and $k_0$, we may focus first on the couple associated to the maximal value for $\imag{ \partial_\nu \lambda^+_{k_0} (\ri \xi_0)}$, and return to the others later.

From positiveness of $\imag{ \partial_\nu \lambda^+_{k_0} (\ri \xi_0)}$, we deduce two facts. First, $\nu \mapsto \lambda^+_{k_0}(\nu)$ is locally invertible with smooth inverse. Second, the curve 
\begin{equation}
\label{e:spectral-curve}
\Set{\lambda^+_{k_0}(\ri \xi) : \absolute{\xi - \xi_0} \leq \delta}
\end{equation}
is tangent to the imaginary axis, and directed from $\Set{\imag \lambda < 0}$ to $\Set{\imag \lambda > 0}$ when $\xi$ pass increasingly through $\xi_0$. To use this geometric fact, we recall that $\nu \mapsto \lambda^+_{k_0}(\nu)$ and its inverse conserve angles due to holomorphicity. Thus the equivalence 
\begin{equation*}
\real{\nu^+_{j_0}(\lambda)} > 0 
\hspace{2em}
\iff
\hspace{2em}
\lambda \notin \varSigma_\textrm{ess}(\cL)  
\end{equation*}
holds in a neighborhood of $(\lambda, \nu) = (0, \ri \xi_0)$.\footnote{Recall that $\imag{ \partial_\nu \lambda^+_{k_0} (\ri \xi_0)}$ is maximal.} The sign $\real{\nu^+_{j_0}(\lambda)} > 0$ now holds true for all small $\lambda$ with positive real part, and can be extended to the entire half plane $\Set{\real{\lambda} > 0}$ using the stability of the essential spectrum. This is a contradiction with \eqref{e:Evans-mode}. 

When studying other couples $(\phi, k_0)$, the right hand side of the above equivalence must be replaced with the looser statement that $\lambda$ is at the right of the curve \eqref{e:spectral-curve}. With this change, the proof holds.
\end{proof}

With the extra assumption \eqref{e:incoming-group-velocities}, it is again possible to distinguish between Fourier and Evans modes, and the parameter modulation procedure goes as discussed above. Let us stress that when \eqref{e:incoming-group-velocities} does not hold, even Theorem \ref{t:main} proof cease to apply. Indeed, the condition $\alpha_1 > 0$ in \eqref{e:marginal-lambda-expansion} is crucial for the Gaussian kernel estimate \eqref{e:estimate-gi-rho}.

\appendix

\section{Shifting integration paths}
\label{s:appendix}
We show here how deformation of integration path endpoints is possible.

For every $r \in \R$ and $\gamma \in [\frac\pi2, \pi]$, let
\begin{equation*}
\varGamma(r, \gamma) \deq \Set{r + \cos(\gamma) \absolute{\xi} + i \sin(\gamma) \xi : \xi \in \R}.
\end{equation*}
The vertical path is denoted
\begin{equation*}
\varGamma(r) \deq \varGamma\left(r, \frac{\pi}{2}\right).
\end{equation*}
\begin{assumption}
\label{a:integration-path-shift}
Let $\Omega \subset \C$. Assume there exists $\alpha>0$ and a holomorphic map $F : \Omega \to \C$ satisfying 
\begin{equation*}
\absolute{F(\lambda)} \leq C \absolute{\lambda}^\alpha, 
\hspace{4em} 
\lambda \in \Omega,
\end{equation*}
\end{assumption}
In the following, we fix $t>0$, and we write
\begin{equation*}
I(r, \gamma) \deq \int_{\varGamma(r, \gamma)} e^{\lambda t} F(\lambda) \d \lambda, 
\hspace{4em} 
r \in \R, 
\hspace{1em} 
\gamma \in \left[\frac{\pi}{2}, \pi\right],
\end{equation*}
together with $I(r) \deq I\left(r, \frac{\pi}{2}\right)$.

\begin{lemma}
Assume Assumption \ref{a:integration-path-shift} holds. Further assume there exists $\gamma \in [\frac{\pi}{2}, \pi]$ and $(r_1, r_2) \in \R^2$ such that
\begin{equation*}
\cup_{r \in [r_1, r_2]} \varGamma(r, \gamma) \subset \Omega.
\end{equation*}
If $I(r_1, \gamma)$ converges, then $I(r_1, \gamma) = I(r_2, \gamma)$.
\end{lemma}
\begin{proof} 
For every $R>0$, let 
\begin{equation*}
\varGamma_R(r, \gamma) \deq B\left(\frac{r_1 + r_2}{2}, R\right) \cap \varGamma(r, \gamma).
\end{equation*}
Propagating this subscript notation to stand for the intersection with a radius $R$ ball, and due to $F$ being holomorphic,
\begin{equation*}
I_R(r_1, \gamma) - I_R(r_2, \gamma) = \int_{\mathcal{C}(R, \gamma, r_1, r_2)} e^{\lambda t}F(\lambda) \d \lambda,
\end{equation*}
where $\mathcal{C}$ is a piece-wise integration path of length at most $2\pi \absolute{r_1 - r_2}$, and located outside of $B(0, R)$. In particular,
\begin{equation*}
\absolute{I_R(r_1, \gamma) - I_R(r_2, \gamma)} \leq C \absolute{r_1 - r_2} e^{\max(r_1, r_2) t} R^\alpha
\end{equation*}
Which guaranties that $I_R(r_2, \gamma)$ converges when $R\to +\infty$, and that
\begin{equation*}
0 = \lim_{R \to +\infty} I_R(r_1, \gamma) - I_R(r_2, \gamma) = I(r_1, \gamma) - I(r_2, \gamma).
\end{equation*}
\end{proof}
\begin{lemma}
Assume Assumption \ref{a:integration-path-shift} hods. Further assume there exists $r \leq 0$ and $\gamma \in (\frac{\pi}{2}, \pi]$ such that \begin{equation*}
\cup_{\varphi \in [\frac{\pi}{2}, \gamma]}\varGamma(r, \varphi) \subset \Omega
\end{equation*}
If $I(r)$ converges, then $I(r) = I(r, \gamma)$.
\end{lemma}
\begin{proof}
As above, let
\begin{equation*}
\varGamma_R(r, \gamma) \deq B\left(r, R\right) \cap \varGamma(r, \gamma), 
\hspace{4em} R > 0. 	
\end{equation*}
and
\begin{equation*}
\mathcal{C}_R(r, \gamma) = \Set{r + R e^{i \xi} : \absolute{\xi} \in \left[\frac{\pi}{2}, \gamma\right]}, 
\hspace{4em} R > 0.
\end{equation*}
Since $F$ is holomorphic,
\begin{equation*}
I_R(r) - I_R(r, \gamma) = \int_{\mathcal{C}_R(r, \gamma)} e^{\lambda t} F(\lambda)  \d \lambda.
\end{equation*}
To keep notations simple, we focus in the following on $\mathcal{C}_R(r, \gamma) \cap \Set{\imag(\lambda) \geq 0}$ instead of $\mathcal{C}_R$. Let $\delta = R^{-1/3}$. We decompose $\mathcal{C}_R$ as the union of the following two integration paths:
\begin{equation*}
\mathcal{C}_1 = \Set{r + R e^{i \xi} : \absolute{\xi} \in \left[\frac{\pi}{2} + \delta, \gamma \right]}.
\end{equation*}
\begin{equation*}
\mathcal{C}_2 = \Set{r + R e^{i \xi} : \absolute{\xi} \in \left[\frac{\pi}{2}, \frac{\pi}{2} + \delta \right]}, 
\end{equation*}
On the first hand, $\mathcal{C}_1$ is strictly away from the imaginary axis, thus exponentially decay with $R$:
\begin{equation*}
\Absolute{\int_{\mathcal{C}_1} e^{\lambda t} f(\lambda) \d \lambda} \leq C R^\alpha \int_{\frac{\pi}{2} + \delta}^\gamma e^{R\cos(\xi) t} R \d \xi \leq C R^{1+\alpha} e^{R \cos\left(\frac{\pi}{2} + \delta\right) t} \frac{\pi}{2}.
\end{equation*}
More precisely, since $\cos\left(\frac{\pi}{2} + x \right) = - \sin(x) = - x + g(x)$, where the map $g$ satisfies
\begin{equation*}
0 \leq g(x) \leq C x^3, 
\hspace{4em}
x \in [-\varepsilon_0, \varepsilon_0],
\end{equation*}
the previous choice for $\delta$ ensures 
\begin{equation*}
\Absolute{\int_{\mathcal{C}_1} e^{\lambda t} f(\lambda) \d \lambda} \leq C R^{1+\alpha} e^{- R \delta t + R g(\delta) t} \leq C R^{1 + \alpha} e^{-R^{2/3} t + C t}.
\end{equation*}
The right hand side converges to $0$ when $R \to +\infty$. 
On the other hand, the second integration path length goes to $0$:
\begin{equation*}
\Absolute{\int_{\mathcal{C}_2} e^{\lambda t} F(\lambda) \d \lambda} \leq C R^{1 + \alpha} \int_{0}^\delta e^{-R\sin(\xi) t} \d \xi \leq C R^{1+\alpha} e^{R g(\delta) t} \int_0^\delta e^{-R \xi t} \d \xi.
\end{equation*}
As above, $\delta$ expression ensures $e^{R g(\delta)t} \leq C$, so that the change of variable $R \xi = \zeta$ leads to
\begin{equation*}
\Absolute{\int_{\mathcal{C}_2} e^{\lambda t} F(\lambda) \d \lambda} \leq C R^{1+\alpha} \frac{\delta}{R^2} \leq C R^{\alpha - 4/3}.
\end{equation*}
Since $\alpha < 0$, the right hand side converges to $0$ when $R \to +\infty$. Thus, $I_R(r) - I_R(r, \gamma)$ converges to $0$ when $R\to +\infty$, concluding the proof.
\end{proof}

\bibliography{biblio}

\end{document}